\documentclass[twoside,leqno,10pt, e^A4]{amsart}
\usepackage{amsfonts}
\usepackage{amsmath}
\usepackage{amscd}
\usepackage{amssymb}
\usepackage{amsthm}
\usepackage{amsrefs}
\usepackage{latexsym}
\usepackage{mathrsfs}
\usepackage{bbm}
\usepackage{enumerate}
\usepackage{graphicx}

\usepackage{amsfonts}
\usepackage{amsmath}
\usepackage{amscd}
\usepackage{amssymb}
\usepackage{amsthm}
\usepackage{amsrefs}
\usepackage{latexsym}
\usepackage{mathrsfs}
\usepackage{bbm}
\usepackage{amscd}
\usepackage{amssymb}
\usepackage{amsthm}
\usepackage{amsrefs}
\usepackage{latexsym}
\usepackage{mathrsfs}
\usepackage{bbm}
\usepackage{enumerate}
\usepackage{graphicx}
\usepackage{color}
\begin{document}

\newtheorem{theorem}[subsection]{Theorem}
\newtheorem{proposition}[subsection]{Proposition}
\newtheorem{lemma}[subsection]{Lemma}
\newtheorem{corollary}[subsection]{Corollary}
\newtheorem{conjecture}[subsection]{Conjecture}
\newtheorem{prop}[subsection]{Proposition}
\numberwithin{equation}{section}
\newcommand{\mr}{\ensuremath{\mathbb R}}
\newcommand{\mc}{\ensuremath{\mathbb C}}
\newcommand{\dif}{\mathrm{d}}
\newcommand{\intz}{\mathbb{Z}}
\newcommand{\ratq}{\mathbb{Q}}
\newcommand{\natn}{\mathbb{N}}
\newcommand{\comc}{\mathbb{C}}
\newcommand{\rear}{\mathbb{R}}
\newcommand{\prip}{\mathbb{P}}
\newcommand{\uph}{\mathbb{H}}
\newcommand{\fief}{\mathbb{F}}
\newcommand{\majorarc}{\mathfrak{M}}
\newcommand{\minorarc}{\mathfrak{m}}
\newcommand{\sings}{\mathfrak{S}}
\newcommand{\fA}{\ensuremath{\mathfrak A}}
\newcommand{\mn}{\ensuremath{\mathbb N}}
\newcommand{\mq}{\ensuremath{\mathbb Q}}
\newcommand{\half}{\tfrac{1}{2}}
\newcommand{\f}{f\times \chi}
\newcommand{\summ}{\mathop{{\sum}^{\star}}}
\newcommand{\chiq}{\chi \bmod q}
\newcommand{\chidb}{\chi \bmod db}
\newcommand{\chid}{\chi \bmod d}
\newcommand{\sym}{\text{sym}^2}
\newcommand{\hhalf}{\tfrac{1}{2}}
\newcommand{\sumstar}{\sideset{}{^*}\sum}
\newcommand{\sumprime}{\sideset{}{'}\sum}
\newcommand{\sumprimeprime}{\sideset{}{''}\sum}
\newcommand{\sumflat}{\sideset{}{^\flat}\sum}
\newcommand{\shortmod}{\ensuremath{\negthickspace \negthickspace \negthickspace \pmod}}
\newcommand{\V}{V\left(\frac{nm}{q^2}\right)}
\newcommand{\sumi}{\mathop{{\sum}^{\dagger}}}
\newcommand{\mz}{\ensuremath{\mathbb Z}}
\newcommand{\leg}[2]{\left(\frac{#1}{#2}\right)}
\newcommand{\muK}{\mu_{\omega}}
\newcommand{\thalf}{\tfrac12}
\newcommand{\lp}{\left(}
\newcommand{\rp}{\right)}
\newcommand{\Lam}{\Lambda_{[i]}}
\newcommand{\lam}{\lambda}
\def\L{\fracwithdelims}
\def\om{\omega}
\def\pbar{\overline{\psi}}
\def\phis{\phi^*}
\def\lam{\lambda}
\def\lbar{\overline{\lambda}}
\newcommand\Sum{\Cal S}
\def\Lam{\Lambda}
\newcommand{\sumtt}{\underset{(d,2)=1}{{\sum}^*}}
\newcommand{\sumt}{\underset{(d,2)=1}{\sum \nolimits^{*}} \widetilde w\left( \frac dX \right) }
\newcommand{\sumh}{\sum \omega_f^{-1} }

\newcommand{\hf}{\tfrac{1}{2}}
\newcommand{\af}{\mathfrak{a}}
\newcommand{\Wf}{\mathcal{W}}

\theoremstyle{plain}
\newtheorem{conj}{Conjecture}
\newtheorem{remark}[subsection]{Remark}

\makeatletter
\def\widebreve{\mathpalette\wide@breve}
\def\wide@breve#1#2{\sbox\z@{$#1#2$}%
     \mathop{\vbox{\m@th\ialign{##\crcr
\kern0.08em\brevefill#1{0.8\wd\z@}\crcr\noalign{\nointerlineskip}%
                    $\hss#1#2\hss$\crcr}}}\limits}
\def\brevefill#1#2{$\m@th\sbox\tw@{$#1($}%
  \hss\resizebox{#2}{\wd\tw@}{\rotatebox[origin=c]{90}{\upshape(}}\hss$}
\makeatletter

\title[Bounds for moments of symmetric square $L$-functions]{Bounds for moments of symmetric square $L$-functions}

%%\date{\today}
\author{Peng Gao}
\address{School of Mathematical Sciences, Beihang University, Beijing 100191, P. R. China}
\email{penggao@buaa.edu.cn}
\begin{abstract}
 We study the $2k$-th moment at the central point of the family of symmetric square $L$-functions attached to holomorphic Hecke cusp forms of  level one, weight $\kappa$. We establish sharp lower bounds for all real $k \geq 1/2$ unconditionally. Assuming the truth of the generalized Riemann hypothesis, we also obtain sharp lower bounds for all real $0 \leq k < 1/2$ and sharp upper bounds for all real $k \geq 0$.
\end{abstract}

\maketitle

\noindent {\bf Mathematics Subject Classification (2010)}: 11M06, 11F66, 11F67  \newline

\noindent {\bf Keywords}:  moments, symmetric square $L$-functions, lower bounds, upper bounds

\section{Introduction}
\label{sec 1}

  An important subject to study in number theory is the non-vanishing issue of central values of $L$-functions as these central values carry deep arithmetic information. For example, the well-known Birch and Swinnerton-Dyer conjecture relates the rank of an elliptic curve to the order of vanishing at the central point of the associated $L$-function. It is now a standard procedure in analytic number theory to apply the mollifier method to obtain a positive portion of non-vanishing result concerning a given family of $L$-functions. This approach goes back to the work of A. Selberg \cite{Selberg42}, who showed that a positive proportion of the non-trivial zeros of the Riemann zeta function $\zeta(s)$ lie on the critical line.

  The method of mollifier involves with constructing a short Dirichlet polynomial called a mollifier that approximates the inverse of the corresponding $L$-function.  A classical way to do so is based on the Dirichlet series of the inverse of the given $L$-function and a mollifier is essentially obtained using a truncation of the corresponding Dirichlet series.  In order to achieve a positive portion of non-vanishing result, one then needs to evaluate asymptotically the first and second moments of the family of $L$-functions twisted by the mollifier.  As this process heavily relies on information concerning the second moment, such tactic often becomes ineffective due to the lack of the needed result.

  In \cite{Gonek12}, S. M. Gonek constructed approximations of $\zeta(s)$ and Dirichlet $L$-functions using truncations of their Euler products to study the non-trivial zeros of these functions. In \cite{Radziwill&Sound}, M. Radziwi{\l\l} and K. Soundararajan developed an upper bounds principle to study upper bounds for moments of families of $L$-functions. A dual lower bounds principle was introduced by W. Heap and K. Soundararajan in \cite{H&Sound}. Both principles make crucial use of mollifiers that are constructed based on truncated Euler products of the $L$-functions involved. This way of obtaining mollifiers has been further carried out in the work of S. Lester and M. Radziwi{\l\l} \cite{LR21} to study sign changes of Fourier coefficients of half-integral weight modular forms and in the work of C. David, A. Florea and M. Lalin \cite{DFL21} to establish a positive proportion non-vanishing result of cubic $L$-functions in the function field setting. We point out here that, the lower bounds principle enunciated in \cite{H&Sound} actually allows one to achieve a positive portion of non-vanishing result concerning central values of families of $L$-functions, at least under the assumption of the generalized Riemann hypothesis (GRH). This approach becomes particularly appealing in situations when our understanding on the related second moment is insufficient for one to apply the classical mollifier method. In fact, one may circumvent the evaluation of the second moment using the new mollifiers by incorporating a powerful method of K. Soundararajan \cite{Sound2009} with its refinement by A. J. Harper \cite{Harper} to bound twisted moments of $L$-functions from the above under GRH. This strategy has been employed in \cite{G&Zhao10} to establish a positive portion of non-vanishing result of central values of cubic or  quartic Dirichlet $L$-functions on GRH.

  It is the aim of this paper to apply the new construction of the mollifiers to establish sharp bounds for moments of symmetric square $L$-functions concerning the central values of the family of symmetric square $L$-functions attached to holomorphic Hecke cusp forms as well as sharp bounds for moments at the central point of the same family. To state our result, we denote $H_{\kappa}$ for the set of holomorphic cusp forms $f$ of level one and weight $\kappa$ that are
eigenfunctions of every Hecke operator. Then the Fourier expansion of any $f \in H_{\kappa}$ at infinity can be written as
\begin{align}
\label{fFourier}
f(z) = \sum_{n=1}^{\infty} \lambda_f (n) n^{\frac{\kappa -1}{2}} e(nz),
\end{align}
  where we denote $e(z)$ for $e^{2 \pi i z}$.

 The symmetric square $L$-function of $f$ is then defined for $\Re(s)>1$ by (see \cite[p. 137]{iwakow} and \cite[(25.73)]{iwakow})
\begin{align*}
 L(s, \operatorname{sym}^2 f)=\zeta(2s) \sum_{n \geq 1}\frac {\lambda(n^2)}{n^s}.
\end{align*}

 An asymptotic formula for the first moment of $L(1/2, \operatorname{sym}^2 f)$ averaged over $H_{\kappa}$ as $\kappa \rightarrow \infty$ was obtained by  Y.-K. Lau \cite{Lau02} and was subsequently improved in \cites{Fomenko03, Khan07, Sun13, Ng17, Liu18, B&F18}. However, it remains a
 challenge problem to evaluate the second moment of the above family of $L$-functions asymptotically. Various upper bounds concerning the second moment can be found in \cite{I&M01, Luo12, Lam15}. Moreover, by taking an extra average over the weights, R. Khan \cite{Khan10} evaluated the twisted second moment asymptotically to establish a positive proportion of non-vanishing result of the enlarged family of $L$-functions. An asymptotic formula for the third moment concerning this larger family of symmetric $L$-functions is obtained in \cite{D&K18}.

  In order to treat the moments of central values of symmetric $L$-functions, it is natural to introduce the harmonic sum $\mathop{{\sum}^{h}}_{f \in S}$ over any subset $S \subset H_{\kappa}$ such that
\begin{align*}
%%\label{Lfcnsk}
  \mathop{{\sum}^{h}}_{f \in S}:=\sum_{f \in S} \omega^{-1}_f,
\end{align*}
  where $\omega_f$ is the harmonic weight given by
\begin{align}
\label{omegaf}
 \omega_f=\frac {(4\pi)^{\kappa-1}}{\Gamma(\kappa-1)}<f, f>=\frac {\kappa-1}{2\pi}L(1, \operatorname{sym}^2 f).
\end{align}
  Here $<f, f>$ denotes the Petersson inner product.

  The density conjecture of N. Katz and P. Sarnak \cite{K&S} implies that the distribution of zeros near the central point of a reasonable family of $L$-functions is governed by a symmetry type. In \cite{C&F00}, J.B. Conrey and D. W. Farmer further illustrated that such symmetry type also affects the behaviors of the moments of the corresponding family of $L$-functions. In particular, a conjecture on the $2k$-th moment of  $L(1/2, \operatorname{sym}^2 f)$ over $f \in H_{\kappa}$ is given on \cite[p. 887]{C&F00} for all real $k \geq 0$, which can be stated that (see also \cite[p. 2145]{Tang13}) as $\kappa \rightarrow \infty$,
\begin{align}
\label{Lfcnsk}
  \mathop{{\sum}^{h}}\limits_{f\in H_{\kappa}}|L(\frac 12, \operatorname{sym}^2 f)|^{2k} \sim (\log \kappa)^{k(2k+1)}.
\end{align}

  Following a general method developed by Z. Rudnick and K. Soundararajan in \cite{R&Sound, R&Sound1} towards establishing sharp
lower bounds for moments of $L$-functions, H. Tang \cite{Tang13} obtained lower bounds of the desired order of magnitude for all integers $k \geq 0$ for the family of $L$-functions given in \eqref{Lfcnsk}.

 Our first result in this paper extends the result in \cite{Tang13} to all real $k \geq 0$.
\begin{theorem}
\label{thmlowerbound}
    With the notation as above. For any real number $k \geq 1/2$, we have
\begin{align}
\label{lowerbound}
   \mathop{{\sum}^{h}}\limits_{f\in H_{\kappa}}|L(\frac 12, \operatorname{sym}^2 f)|^{2k} \gg_k (\log \kappa)^{k(2k+1)}.
\end{align}
  The above bounds continue to hold for $0 \leq k <1/2$ provided that we assume the truth of GRH.
\end{theorem}

    We achieve the proof of Theorem \ref{thmlowerbound} by applying the lower bounds principle of W. Heap and K. Soundararajan \cite{H&Sound}. We point out here that the case $k=0$ of \eqref{lowerbound} in fact yields a positive portion of non-vanishing result on central values of the family of symmetric square $L$-functions under GRH. However, a superior result with an explicit proportion has already been achieved by H. Iwaniec, W. Luo and P. Sarnak in \cite[Corollary 1.8]{ILS}, by computing the one level density of low-lying zeros of the family of symmetric square $L$-functions.

  Also as indicated above, the method of K. Soundararajan \cite{Sound2009} with its refinement by A. J. Harper \cite{Harper} on upper bounds for moments of $L$-functions under GRH plays a key role in the proof of Theorem \ref{thmlowerbound}. The same method further allows us to obtain the next result concerning the upper bounds of the corresponding $L$-functions.
\begin{theorem}
\label{thmupperbound}
    With the notation as above and assuming the truth of GRH. For any real number $k \geq 0$, we have
\begin{align*}
%%\label{upperbound}
   \mathop{{\sum}^{h}}\limits_{f\in H_{\kappa}}|L(\frac 12, \operatorname{sym}^2 f)|^{2k}  \ll_k (\log \kappa)^{k(2k+1)}.
\end{align*}
\end{theorem}

  Now, Theorem \ref{thmlowerbound} and Theorem \ref{thmupperbound} together implies the following result on the order of magnitude of the moments of the family of symmetric square $L$-functions.
\begin{theorem}
\label{thmorderofmag}
  With the notation as above and assuming the truth of GRH. For any real number $k \geq 0$, we have
\begin{align*}
%%\label{orderofmag}
   \mathop{{\sum}^{h}}\limits_{f\in H_{\kappa}}|L(\frac 12, \operatorname{sym}^2 f)|^{2k}  \asymp  (\log \kappa)^{k(2k+1)}.
\end{align*}
\end{theorem}

  As the proof of Theorem \ref{thmupperbound} can be achieved via a variant of the proof of Proposition \ref{Prop6} below, which is part of our treatments of Theorem \ref{thmlowerbound},  we shall omit the proof of Theorem \ref{thmupperbound} and focus on the proof of Theorem \ref{thmlowerbound} in the rest of the paper.

\section{Preliminaries}
\label{sec 2}

\subsection{Sums over primes}
 We reserve the letter $p$ for a prime number throughout the paper and we recall the following result from parts (d) and (b) in \cite[Theorem 2.7]{MVa} concerning sums over primes.
\begin{lemma}
\label{RS} Let $x \geq 2$. We have for some constant $b$,
\begin{align*}
%%\label{M1}
\sum_{p\le x} \frac{1}{p} = \log \log x + b+ O\Big(\frac{1}{\log x}\Big).
\end{align*}
 Moreover, we have
\begin{align*}
%%\label{M2}
\sum_{p\le x} \frac {\log p}{p} = \log x + O(1).
\end{align*}
\end{lemma}

\subsection{Symmetric square $L$-functions}

  We recall here some general results concerning Rankin-Selberg $L$-functions. Let $L(s, \pi)$ be the $L$-function attached to an automorphic cuspidal representation $\pi$ of $\text{GL}_{M}$ over $\mq$. For $\Re(s)$ large enough, $L(s, \pi)$ can be expressed as an Euler product of the form
\begin{align*}
%%\label{Lpi}
 L(s, \pi)=\prod_pL(s, \pi_p)=\prod_p\prod^M_{j=1}(1-\alpha_{\pi}(p, j)p^{-s})^{-1}.
\end{align*}

    The Rankin-Selberg $L$-function $L(s, \pi \otimes \pi)$ associated to $\pi$ is then defined for $\Re(s)$ large enough as a product of local factors $L(s, \pi \otimes \pi)=\prod\limits_{p}L(s, \pi_p \otimes \pi_p)$ such that (see \cite[(2.18)]{R&S}) at unramified $\pi_p$, we have
\begin{align*}
%%\label{RankinSelbergLfcnunramified}
 L(s, \pi_p \otimes \pi_p)=\prod_{1 \leq i, j \leq M}(1-\alpha_{\pi}(p, i)\alpha_{\pi}(p, j)p^{-s})^{-1}.
\end{align*}

  Furthermore, $L(s, \pi \otimes \pi)$ factors as the product of the symmetric and exterior square $L$-functions (see \cite[p. 139]{BG}):
\begin{align*}
%%\label{RankinSelbergdecomp}
 L(s, \pi \otimes \pi)=L(s, \vee^2)L(s, \wedge^2),
\end{align*}
   where we have
\begin{align}
\label{LEulerprod}
  L(s, \vee^2) = \prod_p L_p(s, \vee^2), \quad L(s, \wedge^2)= \prod_p L_p(s, \wedge^2).
\end{align}

  The Euler product for $L(s, \vee^2)$ (and hence $L(s, \wedge^2)$) can be found on \cite[p. 167]{BG}. For unramified $p$, we have
\begin{align}
\label{Lpexp}
  L_p(s, \vee^2) = \prod_{1 \leq i \leq j \leq M}(1-\alpha_{\pi}(p, i)\alpha_{\pi}(p, j)p^{-s})^{-1}, \quad L_p(s, \wedge^2)= \prod_{1 \leq i < j \leq M}(1-\alpha_{\pi}(p, i)\alpha_{\pi}(p, j)p^{-s})^{-1}.
\end{align}

   It is known that $L(s, \pi \otimes \pi)$ has a simply pole at $s = 1$ if and only if $\pi$ is self-contragredient. In which case, the pole of $L(s, \pi \otimes \pi)$ is carried by exactly one of the two factors $L(s, \vee^2)$ or $L(s, \wedge^2)$. We then write the order of the pole of $L(s, \wedge^2)$ as $(\delta(\pi) + 1)/2$  so that $\delta(\pi) = \pm 1$. The order of the pole of $L(s, \vee^2)$ then equals to $(1-\delta(\pi) )/2$.

  Upon taking logarithmic derivatives on both sides of the expressions given in \eqref{LEulerprod} via making use of \eqref{Lpexp}, we deduce from  \cite[Theorem 5.15]{iwakow} that under GRH, we have
\begin{align}
\label{lambdaprelation}
\begin{split}
   \sum_{p \leq x} \sum_{1 \leq i \leq j \leq M}\alpha_{\pi}(p, i)\alpha_{\pi}(p, j)\log p=& \frac {1-\delta(\pi)}{2}x+O(x^{1/2}\log x \log (q_{\operatorname{sym}^2 (\pi)}x)), \\
  \sum_{p \leq x} \sum_{1 \leq i < j \leq M}\alpha_{\pi}(p, i)\alpha_{\pi}(p, j)\log p=& \frac {1+\delta(\pi)}{2}x+O(x^{1/2}\log x \log (q_{\operatorname{ext}^2 (\pi)}x)),
\end{split}
\end{align}
  where $q_{\operatorname{sym}^2 (\pi)}, q_{\operatorname{ext}^2 (\pi)}$ are analytic conductors of $L(s, \vee^2)$ and $L(s, \wedge^2)$ defined on \cite[p. 95]{iwakow}.  Here the implied constants are absolute.

\subsection{Cusp form $L$-functions}
\label{sec:cusp form}

  Recall that the Fourier expansion of any $f \in H_{\kappa}$ at infinity is given in \eqref{fFourier}. For $\Re(s) > 1$, the associated modular $L$-function $L(s, f)$ is defined by
\begin{align}
\label{Lfdef}
L(s, f) &= \sum_{n=1}^{\infty} \frac{\lambda_f (n)}{n^s}
 = \prod_{p} \prod^2_{j =1} (1-\alpha_{f}(p,j)p^{-s})^{-1}.
\end{align}

  By Deligne's proof \cite{D} of the Weil conjecture, we know that
\begin{align}
\label{alpha}
|\alpha_{f}(p,1)|=|\alpha_{f}(p,2)|=1, \quad \alpha_{f}(p,1)\alpha_{f}(p,2)=1.
\end{align}

  It follows from  this and \eqref{Lfdef} that $\lambda_f(n)$ is multiplicative and for $\nu \geq 1$, we have
\begin{align}
\label{lambdafpnu}
 \lambda_{f}(p^{\nu})=\sum^{\nu}_{j=0}\alpha^{v-j}_f(p,1)\alpha^j_{f}(p,2).
\end{align}
  The above relation further implies that $\lambda_f (n) \in \mr$ satisfying $\lambda_f (1) =1$ and $|\lambda_f (n)|
\leq d(n)$ for $n \geq 1$ with $d(n)$ denotes the divisor function of $n$. Moreover, for integers $m, n \geq 1$, we have the following multiplicative property
\begin{align}
\label{lambdamult}
 \lambda_{f}(m)\lambda_f(n)=\sum_{d|(m,n)}\lambda_f(\frac {mn}{d^2}).
\end{align}

 The symmetric square $L$-function $L(s, \operatorname{sym}^2 f)$ of $f$ is then defined for $\Re(s)>1$ by
 (see \cite[p. 137]{iwakow} and \cite[(25.73)]{iwakow})
\begin{align}
\label{Lsymexp}
\begin{split}
 L(s, \operatorname{sym}^2 f)=& \prod_p\prod_{1 \leq i \leq j \leq 2}(1-\alpha_{f}(p, i)\alpha_{f}(p, j)p^{-s})^{-1}=(1-\alpha^2_{f}(p, 1)p^{-s})^{-1}(1-p^{-s})^{-1}(1-\alpha^2_{f}(p, 2)p^{-s})^{-1} \\
    =& \zeta(2s) \sum_{n \geq 1}\frac {\lambda(n^2)}{n^s}=\prod_{p}(1-\frac {\lambda(p^2)}{p^s}+\frac {\lambda(p^2)}{p^{2s}}-\frac {1}{p^{3s}} )^{-1}.
\end{split}
\end{align}

  It follows from a result of G. Shimura \cite{Shimura} that the corresponding completed $L$-function
\begin{align}
\label{Lambdafdef}
 \Lambda(s, \operatorname{sym}^2 f)=& \pi^{-3s/2}\Gamma (\frac {s+1}{2})\Gamma (\frac {s+\kappa-1}{2}) \Gamma (\frac {s+\kappa}{2}) L(s, \operatorname{sym}^2 f)
\end{align}
  is entire and satisfies the functional equation
$\Lambda(s, \operatorname{sym}^2 f)=\Lambda(1-s, \operatorname{sym}^2 f)$.

  Noticing that the weight $\omega_f$ introduced in \eqref{omegaf} satisfies that $\omega_f \neq 0$ so that we have $L(s, \operatorname{sym}^2 f)$ has no pole at $s=1$. Moreover, it follows from \eqref{Lambdafdef} and the definition of the analytic conductor given on \cite[p. 95]{iwakow} that the analytic conductor of $L(s, \operatorname{sym}^2 f)$ is $ \ll \kappa^2$. We then apply \eqref{lambdaprelation} to the representation whose associated $L$-function equals $L(s, f)$ to see by \eqref{Lsymexp} that
\begin{align}
\label{lambdaprelationf0}
\begin{split}
  \sum_{p \leq x} \big(\alpha^2_{f}(p, 1)+\alpha^2_{f}(p, 2)+1 \big )\log p=& O\big(x^{1/2}\log x \log (\kappa x)\big), \\
  \sum_{p \leq x}  \big(\alpha^2_{f}(p, 1)+\alpha^2_{f}(p, 2) \big )\log p=& -x+O\big(x^{1/2}\log x \log (\kappa x)\big).
\end{split}
\end{align}

  We now denote  $\pi_{\operatorname{sym}^2 f}$ for the representation whose associated $L$-function equals $L(s, \operatorname{sym}^2 f)$ and we notice that this representation is also self-contragredient. We also denote the symmetric square and exterior functions associated to $\pi_{\operatorname{sym}^2 f}$ by $L(s, \vee^2;\pi_{\operatorname{sym}^2 f})$ and $L(s, \wedge^2;\pi_{\operatorname{sym}^2 f})$, respectively. Further note that as $L(s, \operatorname{sym}^2 f)$ is an $L$-function of degree $3$, the analytic conductor of $L(s, \vee^2;\pi_{\operatorname{sym}^2 f})$ divides (see \cite[p. 97]{iwakow}) the sixth power of the analytic conductor of $L(s, \operatorname{sym}^2 f)$. We then apply \eqref{lambdaprelation} to $\pi_{\operatorname{sym}^2 f}$ to see that
\begin{align}
\label{lambdaprelationf}
\begin{split}
  \sum_{p \leq x} \big(\alpha^4_{f}(p, 1)+\alpha^4_{f}(p, 2)+1+\alpha^2_{f}(p, 1)+\alpha^2_{f}(p, 2)+1 \big )\log p=& \frac {1-\delta(\pi_{\operatorname{sym}^2 f})}{2}x+O\big(x^{1/2}\log x \log (\kappa x)\big), \\
  \sum_{p \leq x}  \big(\alpha^2_{f}(p, 1)+\alpha^2_{f}(p, 2)+1 \big )\log p=& \frac {1+\delta(\pi_{\operatorname{sym}^2 f})}{2}x+O\big(x^{1/2}\log x \log (\kappa x)\big).
\end{split}
\end{align}
   We then deduce from the second expression above and \eqref{lambdaprelationf0} that $\delta(\pi_{\operatorname{sym}^2 f})=-1$, which in term implies that the first expression in \eqref{lambdaprelationf} can be written as
\begin{align}
\label{sumsymext}
\begin{split}
  \sum_{p \leq x} \big(\alpha^4_{f}(p, 1)+\alpha^4_{f}(p, 2)+1 \big )\log p=& x+O\big(x^{1/2}\log x \log (\kappa x)\big).
\end{split}
\end{align}
   As the prime number theorem under the Riemann hypothesis implies that (see \cite[Theorem 13.1]{MVa1})
\begin{align*}
%%\label{lambdaprelationf}
\begin{split}
  \sum_{p \leq x} \log p = x+O\big(x^{1/2}\log^2 x \big),
\end{split}
\end{align*}
  we readily deduce from \eqref{lambdaprelationf} and \eqref{sumsymext} that the first expression given in \eqref{lambdaprelationf} can be written as
\begin{align}
\label{lambdaprelationfexplicit}
\begin{split}
  \sum_{p \leq x} \big(\alpha^4_{f}(p, 1)+\alpha^4_{f}(p, 2) \big )\log p=& O\big(x^{1/2}\log x \log (\kappa x)\big).
\end{split}
\end{align}

\subsection{Twisted first moment and the Petersson trace formula}

   We quote the following result from \cite{Ng17} on the twisted first moment of $L(\frac 12, \operatorname{sym}^2 f)$ for $f \in H_{\kappa}$.
\begin{lemma}
\label{Twistedfirstmoment}
   For any $\varepsilon > 0, 1 \leq l \ll \kappa$, we have for some constant $C$,
\begin{align}
\label{lstapproxk}
  \mathop{{\sum}^{h}}\limits_{f\in H_{\kappa}} L(\frac 12, \operatorname{sym}^2 f) \lambda_f(l^2) =
  \frac {1}{\sqrt{l}}(-\log l+C+\log \kappa)+O(l\kappa^{-1/2+\varepsilon}+k^{-1}).
\end{align}
\end{lemma}

   Note that the Petersson trace formula states that (see \cite[(2.7)]{I&M01})
\begin{align*}
%%\label{delta}
\begin{split}
\Delta_{m,n}:=& \mathop{{\sum}^{h}}\limits_{f\in H_{\kappa}} \lambda_f(m)\lambda_f(n)=\delta_{m,n}+2\pi i^{\kappa}\sum_{c \geq 1}\frac {S(m,n;c)}{c}
  J_{\kappa-1}(\frac {4\pi\sqrt{mn}}{c}),
\end{split}
\end{align*}
  where $\delta_{m,n}= 1$ if $m = n$ and is $0$ otherwise, $J_{\kappa-1}(x)$ is the Bessel function and $S(m,n;c)$ is the Kloosterman's sum defined by
\begin{align*}
%%\label{Kloosterman}
 S(m,n;c)=\sumstar_{h \shortmod c} e(\frac {mh+n\overline{h}}{c}),
\end{align*}
  where $\sumstar$ denotes the sum over invertible elements modulo $c$.

  We have the following result (see \cite[Lemma 2.1]{R&Sound}) concerning the size of $\Delta_{m,n}$.
\begin{lemma}
\label{LemDeltaest}
   For $mn \leq \kappa^2/10^4$, we have
\begin{align}
\label{Deltaest}
  \Delta_{m,n}=\delta_{m,n}+O(e^{-\kappa}).
\end{align}
\end{lemma}

\subsection{Upper bound for $\log |L(\frac 12, \operatorname{sym}^2 f)|$ }

  Our next result establishes an upper bound of $\log |L(\frac 12, \operatorname{sym}^2 f)|$ in terms of a sum involving prime powers, following the approach in the proof of \cite[Proposition]{Sound2009}.
\begin{lemma}
\label{lem: logLbound}
With the notation as above and assuming the truth of GRH for $L(s, \operatorname{sym}^2 f)$. Let $x \geq 2$ and let $\lambda_0=0.4912\ldots$
denote the unique positive real number satisfying $e^{-\lambda_0} = \lambda_0+\frac {\lam^2_0}{2}$.
We have for any $\lambda \geq \lam_0$,
\begin{align}
\label{logLupperbound}
\log |L(\frac 12, \operatorname{sym}^2 f)| \le   \sum_{\substack {p^l \leq x \\  l \geq 1}} \frac {\alpha^{2l}_{f}(p, 1)+\alpha^{2l}_{f}(p, 2)+1}{lp^{l(\frac12+ \frac{\lam}{\log x})}} \frac{\log \leg {x}{p^l}}{\log x}
 + (1+\lam)\frac{ \log \kappa}{\log x}+ O\Big( \frac{\lambda}{\log x}\Big).
\end{align}
\end{lemma}
\begin{proof}
 We interpret $\log |L(\frac 12, \operatorname{sym}^2 f)|$ as $-\infty$ when $L(\frac 12, \operatorname{sym}^2 f)=0$, so that we may assume $L(\frac 12, \operatorname{sym}^2 f) \neq 0$ in the remaining of the proof.  Recall that the function $\Lambda(s, \operatorname{sym}^2 f)$ defined in \eqref{Lambdafdef} is analytic in the entire complex plane. As $\Gamma(s)$ has simple poles at the non-positive rational integers (see \cite[\S 10]{Da}), we deduce from \eqref{Lambdafdef} that $L(s, \operatorname{sym}^2 f)$ has simple zeros at $s=-2k-1, -2k+1-\kappa, -2k-\kappa$ for all $k \geq 0$. These zeros are called the trivial zeros of $L(s, \operatorname{sym}^2 f)$. Since we assume GRH, we know that the non-trivial zeros of $L(s, \operatorname{sym}^2 f)$ are precisely the zeros of $\Lambda(s, \operatorname{sym}^2 f)$.

  Let $\rho=\tfrac 12+i\gamma$ with $\gamma \in \mr$ run over the non-trivial zeros of $L(s, \operatorname{sym}^2 f)$. We deduce from \cite[Theorem 5.6]{iwakow} and the observation that $L(s, \operatorname{sym}^2 f)$ is analytic at $s=1$ that
\begin{align}
\label{Lproductzeros}
  \Lambda(s, \operatorname{sym}^2 f)=e^{B_0 + B_1s}\prod_{\rho}\left(1 - \frac{s}{\rho}\right)e^{\frac{s}{\rho}},
\end{align}
where $B_0, B_1=B_1(f)$ are constants.

    Taking the logarithmic derivative on both sides of \eqref{Lproductzeros} and making use of \eqref{Lambdafdef}, we obtain that
\begin{equation*}
%%\label{3}
-\frac{L'}{L}(s, \operatorname{sym}^2 f)=\frac{1}{2} \frac{\Gamma'}{\Gamma}(\frac {s+1}{2}) +\frac{1}{2} \frac{\Gamma'}{\Gamma}(\frac {s+\kappa-1}{2}) +\frac{1}{2} \frac{\Gamma'}{\Gamma}(\frac {s+\kappa}{2})  - \frac{3}{2}\log\pi - B_1 - \sum_{\rho}\left(\frac{1}{s-\rho} + \frac{1}{\rho}\right).
\end{equation*}
    This implies that
\begin{align}
\label{Lderivreal}
-\Re{\frac{L'}{L}(s, \operatorname{sym}^2 f)}=\frac{1}{2} \Re{\frac{\Gamma'}{\Gamma}(\frac {s+1}{2})}+ \frac{1}{2} \Re{\frac{\Gamma'}{\Gamma}(\frac {s+\kappa-1}{2})}+ \frac{1}{2} \Re{\frac{\Gamma'}{\Gamma}(\frac {s+\kappa}{2})}- \frac{3}{2}\log\pi - \Re(B_1) - \sum_{\rho}\Re{\left(\frac{1}{s-\rho} + \frac{1}{\rho}\right)}.
\end{align}

    On the other hand, we note that by \cite[(5.29)]{iwakow},
\begin{equation*}
\Re(B_1)=-\sum_\rho\Re(\rho^{-1}).
\end{equation*}
   Combining the above with \eqref{Lderivreal}, we see that
\begin{align}
\label{LprimeLbound}
\begin{split}
-\Re{\frac{L'}{L}(s, \operatorname{sym}^2 f)} =& \frac{1}{2} \Re{\frac{\Gamma'}{\Gamma}(\frac {s+1}{2})}+ \frac{1}{2} \Re{\frac{\Gamma'}{\Gamma}(\frac {s+\kappa-1}{2})}+ \frac{1}{2} \Re{\frac{\Gamma'}{\Gamma}(\frac {s+\kappa}{2})}- \frac{3}{2}\log\pi - \Re{\sum_{\rho}\left(\frac{1}{s-\rho}\right)}\\
=& \frac{1}{2} \Re{\frac{\Gamma'}{\Gamma}(\frac {s+1}{2})}+ \frac{1}{2} \Re{\frac{\Gamma'}{\Gamma}(\frac {s+\kappa-1}{2})}+ \frac{1}{2} \Re{\frac{\Gamma'}{\Gamma}(\frac {s+\kappa}{2})}+O(1) - F(s),
\end{split}
\end{align}
  where
\begin{align*}
 F(s) = \Re{\sum_{\rho} \frac{1}{s-\rho}}= \sum_{\rho} \frac{\sigma-1/2}{(\sigma-1/2)^2+
 (t-\gamma)^2}.
\end{align*}

   Using the estimation $\frac{\Gamma'}{\Gamma}(s)=\log s+O(|s|^{-1})$ by (6) of \cite[\S 10]{Da}, we derive from \eqref{LprimeLbound} that for $s=\sigma \in \mr$ such that $|\sigma|$ is bounded,  we have
\begin{align}
\label{LprimeLbound1}
\begin{split}
-\Re{\frac{L'}{L}(\sigma, \operatorname{sym}^2 f)}=& \log \kappa +O(1) - F(\sigma).
\end{split}
\end{align}

  Integrating both sides of \eqref{LprimeLbound1} from $s=\half$ to $s=\sigma_0 > \frac 12$ implies that
\begin{align}
\label{Lprimediffbound}
\begin{split}
 \log |L(\frac 12,  \operatorname{sym}^2 f)| - \log |L(\sigma_0,  \operatorname{sym}^2 f)|
 =& \Big(\log \kappa +O(1)\Big) (\sigma_0-\half) -\int_{\half}^{\sigma_0} F(\sigma) d\sigma \\
 =&  (\sigma_0-\half) \Big (\log \kappa  +O(1) \Big)-\half \sum_{\rho} \log \frac {(\sigma_0-\half)^2+\gamma^2}{\gamma^2} \\
 \leq &  (\sigma_0-\half) \Big (\log \kappa +O(1)-\half F(\sigma_0) \Big),
\end{split}
\end{align}
  where the last estimation above follows from the observation that $\log (1+x^2) \geq x^2/(1+x^2)$.

  We further deduce from \eqref{Lsymexp} that for $\Re(s)>1$,
\begin{align*}
%%\label{Lprimediffbound}
\begin{split}
 -\frac{L^{\prime}}{L}(s, \operatorname{sym}^2 f)
 =& \sum_{\substack {p^l \\ l \geq 1}} \frac {\log p (\alpha^{2l}_{f}(p, 1)+\alpha^{2l}_{f}(p, 2)+1)}{p^{ls}}.
\end{split}
\end{align*}

 Upon integrating term by term using the above expression for
 $-\frac{L^{\prime}}{L}(s+w, \operatorname{sym}^2 f)$, we obtain that
\begin{align}
\label{Lprimeexp1}
\begin{split}
 \frac{1}{2\pi i} \int\limits_{(c)} -\frac{L^{\prime}}{L}(s+w, \operatorname{sym}^2 f)
 \frac{x^w}{w^2} dw  = \sum_{\substack {p^l \leq x \\  l \geq 1}} \frac {\log p (\alpha^{2l}_{f}(p, 1)+\alpha^{2l}_{f}(p, 2)+1)}{p^{ls}} \log \leg {x}{p^l},
\end{split}
\end{align}
   where $c>2$ is a large number. Upon moving the line of integration in the above expression to the left and calculating residues, we deduce that
\begin{align}
\label{Lprimeexp2}
\begin{split}
& \frac{1}{2\pi i} \int\limits_{(c)} -\frac{L^{\prime}}{L}(s+w, \operatorname{sym}^2 f)
 \frac{x^w}{w^2} dw \\
=& -\frac{L^{\prime}}{L}(s, \operatorname{sym}^2 f) \log x - \Big(\frac{L^{\prime}}{L}(s, \operatorname{sym}^2 f)\Big)^{\prime}
 -\sum_{\rho} \frac{x^{\rho-s}}{(\rho-s)^2} -\sum_{k=0}^{\infty}
  \Big (\frac{x^{-2k-1-s}}{(2k+s+1)^2}+\frac{x^{-2k+1-\kappa-s}}{(2k+\kappa+s-1)^2}+\frac{x^{-2k-\kappa-s}}{(2k+\kappa+s)^2}\Big ).
 \end{split}
 \end{align}

  Comparing the two expressions given in \eqref{Lprimeexp1} and \eqref{Lprimeexp2}, we derive that for any $x \ge 2$,
 \begin{align}
\label{Lprimeseries}
\begin{split}
 -\frac{L^{\prime}}{L}(s, \operatorname{sym}^2 f)=& \sum_{\substack {p^l \leq x \\  l \geq 1}} \frac {\log p (\alpha^{2l}_{f}(p, 1)+\alpha^{2l}_{f}(p, 2)+1)}{p^{ls}} \frac{\log \leg {x}{p^l}}{\log x} + \frac{1}{\log x} \Big(\frac{L^{\prime}}{L}(s, \operatorname{sym}^2 f)\Big)^{\prime}
 + \frac{1}{\log x} \sum_{\rho} \frac{x^{\rho-s}}{(\rho-s)^2} \\
& + \frac{1}{\log x} \sum_{k=0}^{\infty} \Big (\frac{x^{-2k-1-s}}{(2k+s+1)^2}+\frac{x^{-2k+1-\kappa-s}}{(2k+\kappa+s-1)^2}+\frac{x^{-2k-\kappa-s}}{(2k+\kappa+s)^2}\Big ).
\end{split}
 \end{align}

 Upon integrating the real parts on both sides of \eqref{Lprimeseries} from $s=\sigma_0$ to $\infty$, we derive that for $x \geq 2$,
 \begin{align}
\label{logL}
\begin{split}
\log |L(\sigma_0, \operatorname{sym}^2 f)| = \Re \Big( \sum_{\substack {p^l \leq x \\  l \geq 1}} \frac {\alpha^{2l}_{f}(p, 1)+\alpha^{2l}_{f}(p, 2)+1}{lp^{l\sigma_0}} \frac{\log \leg {x}{p^l}}{\log x}
  &- \frac{1}{\log x} \frac{L^{\prime}}{L}(\sigma_0, \operatorname{sym}^2 f)\\
  & + \frac{1}{\log x} \sum_{\rho}
 \int_{\sigma_0}^{\infty} \frac{x^{\rho-s}}{(\rho-s)^2} d\sigma +O\Big(\frac{1}{\log x}\Big)\Big).
\end{split}
\end{align}
 Observe that
$$
\sum_{\rho}\Big|\int_{\sigma_0}^{\infty} \frac{x^{\rho -s}}{(\rho -s)^2} d\sigma\Big|
\le \sum_{\rho}\int_{\sigma_0}^{\infty}\frac{ x^{\frac 12-\sigma}}{|\sigma_0-\rho|^2} d\sigma
= \sum_{\rho}\frac{x^{\frac 12-\sigma_0}}{|\sigma_0-\rho|^2 \log x}= \frac{x^{\frac 12-\sigma_0}F(\sigma_0)}{(\sigma_0-\frac 12)\log x}.
$$
 Applying this and \eqref{LprimeLbound1} in \eqref{logL}, we see that
 \begin{align}
\label{logLbound}
\begin{split}
 \log |L(\sigma_0, \operatorname{sym}^2 f)|
 &\le  \Re \Big( \sum_{\substack {p^l \leq x \\  l \geq 1}} \frac {\alpha^{2l}_{f}(p, 1)+\alpha^{2l}_{f}(p, 2)+1}{lp^{l\sigma_0}} \frac{\log \leg {x}{p^l}}{\log x}
 +  \frac {\log \kappa}{ \log x}  \\
& \hskip .5 in+F(\sigma_0) \Big( \frac{x^{\frac 12-\sigma_0}}{(\sigma_0-\half) \log^2 x} -\frac{1}{\log x}
 \Big)
 + O\Big(\frac{1}{\log x}\Big)\Big).
\end{split}
 \end{align}

 Now, we add the two estimations given in \eqref{Lprimediffbound} and \eqref{logLbound} to see that
\begin{align*}
  \log |L(\frac 12, \operatorname{sym}^2 f)|
 \le &  \Re  \sum_{\substack {p^l \leq x \\  l \geq 1}} \frac {\alpha^{2l}_{f}(p, 1)+\alpha^{2l}_{f}(p, 2)+1}{lp^{l\sigma_0}} \frac{\log \leg {x}{p^l}}{\log x}
 +  \log \kappa  \Big(\sigma_0 -\half + \frac{1}{\log x} \Big )
 \\
 &\hskip .5 in +F(\sigma_0) \Big( \frac{x^{\frac 12-\sigma_0}}{(\sigma_0-\half) \log^2 x} -\frac{1}{\log x}-\half (\sigma_0-\half)
 \Big) + O\Big(\frac{1}{\log x}\Big)+O(\sigma_0-\frac 12) \\
=&   \sum_{\substack {p^l \leq x \\  l \geq 1}} \frac {\alpha^{2l}_{f}(p, 1)+\alpha^{2l}_{f}(p, 2)+1}{lp^{l\sigma_0}} \frac{\log \leg {x}{p^l}}{\log x}
 +  \log \kappa  \Big(\sigma_0 -\half + \frac{1}{\log x} \Big )
 \\
 &\hskip .5 in +F(\sigma_0) \Big( \frac{x^{\frac 12-\sigma_0}}{(\sigma_0-\half) \log^2 x} -\frac{1}{\log x}-\half (\sigma_0-\half)
 \Big) + O\Big(\frac{1}{\log x}\Big)+O(\sigma_0-\frac 12),
 \end{align*}
  where the last estimation above follows by noting that we have $\alpha^{2l}_{f}(p, 1)+\alpha^{2l}_{f}(p, 2)+1 \in \mr$ for all $l \geq 1$ by \eqref{alpha}. By setting $\sigma_0 =\frac12+ \frac{\lam}{\log x}$ and then omitting the term involving with $F(\sigma_0)$ since it is negative, we obtain the assertion of the proposition, this completes the proof.
 \end{proof}

   We want to further simplify the upper bound for $\log |L(\frac 12, \operatorname{sym}^2 f)|$ given in \eqref{logLupperbound}. To do so, we first observe from \eqref{alpha} that the terms on the right-hand side of \eqref{logLupperbound} corresponding to $p^l$ with $l \geq 3$ contribute $O(1)$. Also, by \eqref{lambdafpnu} we have $\alpha^{2}_{f}(p, 1)+\alpha^{2}_{f}(p, 2)+1=\lambda_f(p^2)$. We deduce from these observations by setting $\lam=\lam_0$ in \eqref{logLupperbound} to see that
\begin{align}
\label{logLupperbounduptosquare}
\begin{split}
 & \log |L(\frac 12, \operatorname{sym}^2 f)| \le    \sum_{\substack {p \leq x}} \frac {\lambda_{f}(p^2)}{p^{\frac12+ \frac{\lam_0}{\log x}}} \frac{\log \leg {x}{p}}{\log x}+ \sum_{\substack{  p \leq  x^{1/2} }} \frac {\alpha^{4}_{f}(p, 1)+\alpha^{4}_{f}(p, 2)+1}{2p^{1+ \frac{2\lam_0}{\log x}}} \frac{\log \leg {x}{p^2}}{\log x} + (1+\lam_0)\frac{ \log \kappa}{\log x}+ O( 1 ) .
\end{split}
\end{align}
 Note that by Lemma \ref{RS} and \eqref{alpha}, we have
\begin{align}
\label{estPlogP}
  \sum_{\substack{ p \leq x^{1/2} }}  \frac{\alpha^{4}_{f}(p, 1)+\alpha^{4}_{f}(p, 2)+1}{p^{1+\frac{2\lambda_0}{\log x}}}  \frac{\log p}{\log x} \ll  \frac 1{\log x} \sum_{\substack{ p \leq x^{1/2} }}  \frac{\log p}{p} =O(1).
\end{align}

  Similarly, we have
\begin{align}
\label{estPlogx}
\begin{split}
  \sum_{\substack{  p \leq x^{1/2} }} \Big (\frac{1}{p^{1+2\lambda_0/\log x}}- \frac{1}{p} \Big ) & \ll \frac 1{\log x}\sum_{\substack{  p \leq x^{1/2} }}\frac{\log p}{p} \ll 1.
\end{split}
 \end{align}

   We apply the estimations obtained in \eqref{estPlogP} and \eqref{estPlogx} into \eqref{logLupperbounduptosquare} to deduce that
\begin{align}
\label{logLupperbound1}
\begin{split}
  \log |L(\frac 12, \operatorname{sym}^2 f)|
\le &    \sum_{\substack {p \leq x}} \frac {\lambda_{f}(p^2)}{p^{\frac12+ \frac{\lam_0}{\log x}}} \frac{\log \leg {x}{p}}{\log x}+\sum_{\substack {p \leq x^{1/2}}} \frac {\alpha^{4}_{f}(p, 1)+\alpha^{4}_{f}(p, 2)+1}{2p}
 + (1+\lam_0)\frac{ \log \kappa}{\log x}+ O( 1 ) \\
\le &    \sum_{\substack {p \leq x}} \frac {\lambda_{f}(p^2)}{p^{\frac12+ \frac{\lam_0}{\log x}}} \frac{\log \leg {x}{p}}{\log x}+\sum_{\substack {p \leq x^{1/2}}} \frac {\alpha^{4}_{f}(p, 1)+\alpha^{4}_{f}(p, 2)}{2p}+\frac 12\log \log x
 + (1+\lam_0)\frac{ \log \kappa}{\log x}+ O( 1 ),
\end{split}
\end{align}
  where the last estimation above follows from Lemma \ref{RS}.

 We next apply \eqref{lambdaprelationfexplicit} and partial summation to see that when $x^{1/2} \geq (\log \kappa)^6$,
\begin{align}
\label{sumplarge}
 \sum_{\substack{(\log \kappa)^6 < p \leq x^{1/2} }}  \frac{\alpha^4_{f}(p, 1)+\alpha^4_{f}(p, 2) }{2p}   = O(1).
\end{align}

   Moreover, we deduce from \eqref{alpha} and Lemma \ref{RS} that
\begin{align}
\label{sumpsmall}
\begin{split}
  \sum_{\substack{\log \kappa  \leq p < (\log \kappa)^6}}  \frac{\alpha^4_{f}(p, 1)+\alpha^4_{f}(p, 2)}{p} \ll & \sum_{\substack{\log \kappa \leq p < (\log \kappa)^6}}
  \frac{1}{p}  = O(1), \\
   \sum_{\substack{p \leq \log \kappa}}  \frac{\alpha^4_{f}(p, 1)+\alpha^4_{f}(p, 2)}{p} \ll &
  \sum_{\substack{p \leq \log \kappa}}  \frac{1}{p}  = O(\log \log \log \kappa).
\end{split}
\end{align}

   Combining \eqref{logLupperbound1}-\eqref{sumpsmall}, we readily deduce the following result.
\begin{lemma}
\label{lem: logLbound1}
With the notation as above. Assuming the truth of GRH for $L(s, \operatorname{sym}^2 f)$, $L(s, \vee^2;\pi_{\operatorname{sym}^2 f})$ and $L(s, \wedge^2;\pi_{\operatorname{sym}^2 f})$. Let $x \geq (\log \kappa)^{12}$ and let $\lambda_0=0.4912\ldots$ denote the unique positive real number satisfying $e^{-\lambda_0} = \lambda_0+\frac {\lam^2_0}{2}$.
We have
\begin{align}
\label{logLupperboundsimplified}
\log |L(\frac 12, \operatorname{sym}^2 f)| \le      \sum_{\substack {p \leq x }} \frac {\lambda_f(p^2)}{p^{\frac12+ \frac{\lam_0}{\log x}}} \frac{\log \leg {x}{p}}{\log x}+\sum_{\substack{  p \leq  \log \kappa }} \frac {\alpha^{4}_{f}(p, 1)+\alpha^{4}_{f}(p, 2)}{2p}
 +\frac 12 \log \log x+ (1+\lam_0)\frac{ \log \kappa}{\log x}+ O(1).
\end{align}
   Also, we have
\begin{align}
\label{equ:3.3'}
\begin{split}
 & \log |L(\frac 12, \operatorname{sym}^2 f)| \leq  \sum_{\substack {p \leq x }} \frac {\lambda_f(p^2)}{p^{\frac12+ \frac{\lam_0}{\log x}}} \frac{\log \leg {x}{p}}{\log x}+\frac 12\log \log x+
 (1+\lam_0)\frac{ \log \kappa}{\log x}
 +O(\log \log \log \kappa).
\end{split}
 \end{align}
\end{lemma}

\section{Outline of the Proof}
\label{sec 2'}

   As mentioned earlier, we shall only consider the proof of Theorem \ref{thmlowerbound} here. Since the case $2k=1$ of \eqref{lowerbound} can be deduced from Lemma \ref{Twistedfirstmoment}, we may assume in our proof that $2k \neq 1$ is a fixed positive real number and that $\kappa$ is a large number.  Without further notice, we point out here that throughout our proof, the explicit constants involved in various estimations using $\ll$ or the big-$O$ notations depend on $k$ only and are uniform with
 respect to $\kappa$. We further make the convention that an empty product is defined to be $1$.

 Following the ideas of A. J. Harper in \cite{Harper},  we define for a large number $M$ depending on $k$ only,
\begin{align}
\label{alphadef}
\begin{split}
 \alpha_{0} = 0, \;\;\;\;\; \alpha_{j} = \frac{20^{j-1}}{(\log\log \kappa)^{2}} \;\;\; \forall \; j \geq 1, \quad
\mathcal{J} = \mathcal{J}_{k,\kappa} = 1 + \max\{j : \alpha_{j} \leq 10^{-M} \} .
\end{split}
\end{align}

   We denote $P_j=(\kappa^{\alpha_{j-1}}, \kappa^{\alpha_{j}}]$ for $1 \leq j \leq \mathcal{J}$. The above notations and Lemma \ref{RS} then imply that for $\kappa$ large enough,
\begin{align}
\label{sumpj}
\begin{split}
 \sum_{p  \in P_{1}} \frac{1}{p} \leq & \log\log \kappa=\alpha^{-1/2}_1,  \\
 \sum_{ p \in P_{j+1}} \frac{1}{p}
 =& \log \alpha_{j+1} - \log \alpha_{j} + o(1) =  \log 20 + o(1) \leq 10, \quad 1 \leq j \leq \mathcal{J}-1.
\end{split}
\end{align}

  For any real number $x$, we denote $\lceil x \rceil$ to be $\min \{ n \in \intz : n \geq x\}$. We then define a sequence of even natural
  numbers $\{ \ell_j \}_{1 \leq j \leq \mathcal{J}}$ such that $\ell_j =2\lceil e^A\alpha^{-3/4}_j \rceil$, where $A$ is a large number depending on $k$ only. We also define for $1 \leq j \leq \mathcal{J}$,
\begin{align}
\label{defP}
{\mathcal P}_j(f) = \sum_{p\in P_j} \frac{\lambda_f(p^2)}{\sqrt{p}}, \quad  {\mathcal Q}_j(f, 2k) =\Big (\frac{c_k {\mathcal
P}_j(f) }{\ell_j}\Big)^{r_k\ell_j},
\end{align}
  where
\begin{align}
\label{r_k}
\begin{split}
  c_k= & 64 \max (1, 2k), \\
r_k =& \left\{
 \begin{array}
  [c]{ll}
   2\lceil k/(k-1)\rceil & 2k>1,\\
  2\lceil (2-3k)/(1-2k)\rceil+2 & 2k<1.
 \end{array}
 \right.
 \end{split}
\end{align}

  For any non-negative integer $\ell$ and any real number $x$, we denote
\begin{equation*}
%%\label{E_ell}
E_{\ell}(x) = \sum_{j=0}^{\ell} \frac{x^{j}}{j!}.
\end{equation*}
  We further apply the above notations to define for each $1 \leq j \leq \mathcal{J}$ and any real number $\alpha$,
\begin{align*}
%%\label{defN}
{\mathcal N}_j(f, \alpha) = E_{\ell_j} (\alpha {\mathcal P}_j(f)), \quad \mathcal{N}(f, \alpha) = \prod_{j=1}^{\mathcal{J}} {\mathcal
N}_j(f,\alpha).
\end{align*}

   Our lemma below adapts the lower bounds principle of W. Heap and K. Soundararajan \cite{H&Sound} in our setting.
\begin{lemma}
\label{lem1}
 With the notation as above. For $0<k<1/2$, we have
\begin{align}
\label{basiclowerbound1}
\begin{split}
 & \mathop{{\sum}^{h}}\limits_{f\in H_{\kappa}} L(\frac 12, \operatorname{sym}^2 f)  \mathcal{N}(f, 2k-1) \\
 \leq & \Big ( \mathop{{\sum}^{h}}\limits_{f\in H_{\kappa}}|L(\frac 12, \operatorname{sym}^2 f)|^{2k} \Big )^{1/(2(2-3k))}
 \Big ( \mathop{{\sum}^{h}}\limits_{f\in H_{\kappa}} |L(\frac 12, \operatorname{sym}^2 f)|^2\mathcal{N}(f, 2k-2) \Big)^{(1-2k)/(2-3k)} \\
 & \times \Big ( \mathop{{\sum}^{h}}\limits_{f\in H_{\kappa}}  \prod^{\mathcal{J}}_{j=1} \Big ( {\mathcal N}_j(f, 2k) +  {\mathcal Q}_j(f, 2k)   \Big )  \Big)^{(1-2k)/(2(2-3k))}.
\end{split}
\end{align}
  Moreover, we have
\begin{align}
\label{basiclowerboundk0}
\begin{split}
  \mathop{{\sum}^{h}}\limits_{f\in H_{\kappa}} L(\frac 12, \operatorname{sym}^2 f)  \mathcal{N}(f, -1)
 \leq & \Big ( \mathop{{\sum}^{h}}\limits_{\substack{f \in H_{\kappa} \\ L(\frac 12, \operatorname{sym}^2 f) \neq 0}}1\Big )^{1/4}
 \Big ( \mathop{{\sum}^{h}}\limits_{f\in H_{\kappa}} |L(\frac 12, \operatorname{sym}^2 f)|^2\mathcal{N}(f, -2) \Big)^{1/2} \\
 & \times \Big ( \mathop{{\sum}^{h}}\limits_{f\in H_{\kappa}}  \prod^{\mathcal{J}}_{j=1} \Big ( {\mathcal N}_j(f, 0) +  {\mathcal Q}_j(f, 0)   \Big )  \Big)^{1/4}.
\end{split}
\end{align}

 For $k>1/2$, we have
\begin{align}
\label{basicboundkbig1}
\begin{split}
 &  \mathop{{\sum}^{h}}\limits_{f\in H_{\kappa}} L(\frac 12, \operatorname{sym}^2 f)  \mathcal{N}(f, 2k-1)
 \leq \Big (  \mathop{{\sum}^{h}}\limits_{f\in H_{\kappa}} |L(\frac 12, \operatorname{sym}^2 f)|^{2k} \Big )^{\frac {1}{2k}}
 \Big (  \mathop{{\sum}^{h}}\limits_{f\in H_{\kappa}}  \prod^{\mathcal{J}}_{j=1} \Big ( {\mathcal N}_j(f, 2k) +  {\mathcal Q}_j(f, 2k)   \Big ) \Big)^{\frac {2k-1}{2k}}.
 \end{split}
\end{align}
  The implied constants in \eqref{basiclowerbound1}-\eqref{basicboundkbig1} depend on $k$ only.
\end{lemma}
\begin{proof}
   We recall from Section \ref{sec:cusp form} that $\lambda_f(n) \in \mr$ for $f \in H_{\kappa}$. This implies that we have ${\mathcal N}_j(f, \alpha) \in \mr$ for all $1 \leq j \leq \mathcal{J}$ so that ${\mathcal N}(f, \alpha) \in \mr$. Moreover, as $\ell_j$ are even integers for all $1 \leq j \leq \mathcal{J}$, it follows from \cite[Lemma 1]{Radziwill&Sound} that ${\mathcal N}(f, \alpha)>0$. Further, it follows from \cite[Lemma 4.1]{Gao2021-2} that we have
\begin{align*}
%%\label{Nprod}
 \mathcal{N}(f, \alpha)\mathcal{N}(f, -\alpha)  \geq 1.
\end{align*}
  We deduce from this that for $0<k<1/2$ and any real $c$ such that $0<c<1$,
\begin{align}
\label{upperboundc}
\begin{split}
 & \mathop{{\sum}^{h}}\limits_{f\in H_{\kappa}} L(\frac 12, \operatorname{sym}^2 f)  \mathcal{N}(f, 2k-1)
  \leq \mathop{{\sum}^{h}}\limits_{f\in H_{\kappa}} |L(\frac 12, \operatorname{sym}^2 f)|  \mathcal{N}(f, 2k-1) \\
 \leq & \mathop{{\sum}^{h}}\limits_{f\in H_{\kappa}} |L(\frac 12, \operatorname{sym}^2 f)|^{c}
 \cdot |L(\frac 12, \operatorname{sym}^2 f)|^{1-c} \mathcal{N}(f, 2k-2)^{(1-c)/2}  \cdot
 \mathcal{N}(f,2k-1) \mathcal{N}(f, 2-2k)^{(1-c)/2}.
\end{split}
\end{align}

  Applying H\"older's inequality with exponents $2k/c, 2/(1-c), ((1+c)/2-c/(2k))^{-1}$ to the last sum above, we derive that
\begin{align*}
%%\label{basicbound1}
\begin{split}
 & \mathop{{\sum}^{h}}\limits_{f\in H_{\kappa}} L(\frac 12, \operatorname{sym}^2 f)  \mathcal{N}(f, 2k-1) \\
 \leq & \Big ( \mathop{{\sum}^{h}}\limits_{f\in H_{\kappa}}|L(\frac 12, \operatorname{sym}^2 f)|^{2k} \Big )^{c/(2k)}
 \Big ( \mathop{{\sum}^{h}}\limits_{f\in H_{\kappa}} |L(\frac 12, \operatorname{sym}^2 f)|^2\mathcal{N}(f, 2k-2) \Big)^{(1-c)/2} \\
 & \times \Big ( \mathop{{\sum}^{h}}\limits_{f\in H_{\kappa}}  \mathcal{N}(f, 2k-1)^{((1+c)/2-c/(2k))^{-1}}\mathcal{N}(f, 2-2k)^{(1-c)/2
 \cdot ((1+c)/2-c/(2k))^{-1}}  \Big)^{(1+c)/2-c/(2k)}.
\end{split}
\end{align*}

  We now set
$$(1-c)/2 \cdot ((1+c)/2-c/(2k))^{-1}=2.$$
   This implies that $c=(\frac 2k-3)^{-1}$ and that
$$((1+c)/2-c/(2k))^{-1}=\frac {2(2-3k)}{1-2k}.$$
   One checks that the above value of $c$ does satisfy that $0<c<1$ when $0<k<1/2$. We thus conclude that when $0<k<1/2$, we have
\begin{align}
\label{basiclowerbound}
\begin{split}
 & \mathop{{\sum}^{h}}\limits_{f\in H_{\kappa}} L(\frac 12, \operatorname{sym}^2 f)  \mathcal{N}(f, 2k-1) \\
 \leq & \Big ( \mathop{{\sum}^{h}}\limits_{f\in H_{\kappa}}|L(\frac 12, \operatorname{sym}^2 f)|^{2k} \Big )^{1/(2(2-3k))}
 \Big ( \mathop{{\sum}^{h}}\limits_{f\in H_{\kappa}} |L(\frac 12, \operatorname{sym}^2 f)|^2\mathcal{N}(f, 2k-2) \Big)^{(1-2k)/(2-3k)} \\
 & \times \Big ( \mathop{{\sum}^{h}}\limits_{f\in H_{\kappa}}  \mathcal{N}(f, 2k-1)^{(2(2-3k))/(1-2k)}\mathcal{N}(f, 2-2k)^2  \Big)^{(1-2k)/(2(2-3k))}.
\end{split}
\end{align}

  Similarly, we set $k=c=0$ in \eqref{upperboundc} and apply H\"older's inequality with exponents $4, 2, 4$ to the last sum there to see that
\begin{align}
\label{upperboundk=0}
\begin{split}
  \mathop{{\sum}^{h}}\limits_{f\in H_{\kappa}} L(\frac 12, \operatorname{sym}^2 f)  \mathcal{N}(f, -1)
 \leq & \Big (  \mathop{{\sum}^{h}}\limits_{\substack{f \in H_{\kappa} \\ L(\frac 12, \operatorname{sym}^2 f) \neq 0}}1 \Big )^{1/4}
 \Big ( \mathop{{\sum}^{h}}\limits_{f\in H_{\kappa}} |L(\frac 12, \operatorname{sym}^2 f)|^2\mathcal{N}(f, -2) \Big)^{1/2}  \Big ( \mathop{{\sum}^{h}}\limits_{f\in H_{\kappa}}  \mathcal{N}(f, -1)^{4}\mathcal{N}(f, -2)^{2}  \Big)^{1/4}.
\end{split}
\end{align}

  An analogue procedure also allows us to deduce that when $k>1/2$, we have
\begin{align}
\label{basicboundkbig}
\begin{split}
 &  \mathop{{\sum}^{h}}\limits_{f\in H_{\kappa}} L(\frac 12, \operatorname{sym}^2 f)  \mathcal{N}(f, 2k-1)
 \leq  \Big (  \mathop{{\sum}^{h}}\limits_{f\in H_{\kappa}} |L(\frac 12, \operatorname{sym}^2 f)|^{2k} \Big )^{\frac {1}{2k}}
 \Big (  \mathop{{\sum}^{h}}\limits_{f\in H_{\kappa}} \mathcal{N}(f, 2k-1)^{\frac {2k}{2k-1}}  \Big)^{\frac {2k-1}{2k}}.
 \end{split}
\end{align}

  We further simplify the right-hand side expressions in \eqref{basiclowerbound}-\eqref{basicboundkbig}
by noticing that the Taylor formula with integral remainder implies that for any $z \in \mc$,
\begin{align*}
\begin{split}
 \Big|e^z-\sum^{d-1}_{j=0}\frac {z^j}{j!}\Big| =& \Big|\frac 1{(d-1)!}\int^z_0e^t(z-t)^{d-1}dt\Big|=\Big|\frac {z^d}{(d-1)!}\int^1_0 e^{zs}(1-s)^{d-1}ds\Big| \\
\leq & \frac {|z|^d}{d!}\max (1, e^{\Re z}) \leq \frac {|z|^d}{d!}e^z \max (e^{-z}, e^{\Re z-z}) \leq \frac {|z|^d}{d!}e^z e^{|z|}.
\end{split}
 \end{align*}

   We derive from this that
\begin{align}
\label{ezrelation}
\begin{split}
  \sum^{d-1}_{j=0}\frac {z^j}{j!} =e^z(1+O(\frac {|z|^d}{d!}e^{|z|})).
\end{split}
 \end{align}

 Now, Stirling's formula (see \cite[(5.112)]{iwakow}) implies that
\begin{align}
\label{Stirling}
\begin{split}
 (\frac me)^m \leq m! \ll \sqrt{m} ( \frac {m }{e} )^{m}.
\end{split}
\end{align}

   We set $z=\alpha{\mathcal P}_j(f), d=\ell_j+1$ in \eqref{ezrelation} and apply \eqref{Stirling} to deduce that
\begin{align*}
%%\label{Stirling}
\begin{split}
 {\mathcal N}_j(f, \alpha) =\exp ( \alpha{\mathcal P}_j(f) )\left( 1+ O\Big (\exp ( |\alpha {\mathcal P}_j(f)| )
 \left( \frac{e|\alpha {\mathcal P}_j(f)|}{ \ell_j+1} \right)^{ \ell_j+1}  \Big)  \right).
\end{split}
\end{align*}

It follows from the above that when
  $|{\mathcal P}_j(f)| \le \ell_j/(20(1+|\alpha|))$,
\begin{align}
\label{Njest1}
{\mathcal N}_j(f, \alpha)  = \exp ( \alpha {\mathcal P}_j(f)  ) \left( 1+   O(e^{-\ell_j}) \right).
\end{align}

  We apply \eqref{Njest1} to see that when $0 \leq k<1/2$ and $|{\mathcal P}_j(f)| \le \ell_j/60$,
\begin{align}
\label{est1}
{\mathcal N}_j(f, 2k-1)^{\frac {2(2-3k)}{1-2k}} {\mathcal N}_j(f, 2-2k)^{2}
&= \exp( 2k {\mathcal P}_j(f))\Big( 1+ O\big( e^{-\ell_j} \big) \Big)^{\frac {2(3-5k)}{1-2k}} = {\mathcal N}_j(f, 2k)
\Big( 1+ O\big(e^{-\ell_j} \big) \Big)^{\frac {5-8k}{1-2k}}.
\end{align}

  Moreover, we note that when $0\leq k<1/2$ and $|{\mathcal P}_j(f)| > \ell_j/60$,
\begin{align}
\label{N22kbound}
\begin{split}
|{\mathcal N}_j(f, 2-2k)| &\le \sum_{r=0}^{\ell_j} \frac{|2{\mathcal P}_j(f)|^r}{r!} \le
|{\mathcal P}_j(f)|^{\ell_j} \sum_{r=0}^{\ell_j} \Big( \frac{60}{\ell_j}\Big)^{\ell_j-r} \frac{2^r}{r!}   \le \Big( \frac{64 |{\mathcal P}_j(f)|}{\ell_j}\Big)^{\ell_j} .
\end{split}
\end{align}
  Notice that the above estimation also holds for $|{\mathcal N}_j(f, 2k-1)|$, so that these estimations and the definition of ${\mathcal Q}_j(f, 2k)$
given in \eqref{defP} imply that when $0\leq k<1/2$ and $|{\mathcal P}_j(f)| > \ell_j/60$,
\begin{align}
\label{est2}
{\mathcal N}_j(f, 2k-1)^{\frac {2(2-3k)}{1-2k}} {\mathcal N}_j(f, 2-2k)^{2}
&\leq  {\mathcal Q}_j(f, 2k).
\end{align}

  We deduce from \eqref{est1} and \eqref{est2} that when $0\leq k<1/2$, we have
\begin{align}
\label{upperboundprodofN}
\begin{split}
 &  \mathop{{\sum}^{h}}\limits_{f\in H_{\kappa}}  \mathcal{N}(f, 2k-1)^{(2(2-3k))/(1-2k)}\mathcal{N}(f, 2-2k)^2 \\
\le &
\mathop{{\sum}^{h}}\limits_{f\in H_{\kappa}}  \Big ( \prod^{\mathcal{J}}_{j=1}
\Big ({\mathcal N}_j(f, 2k)
\Big( 1+ O\big(e^{-\ell_j} \big) \Big)^{\frac {5-8k}{1-2k}}+   {\mathcal Q}_j(f, 2k)  \Big )\Big
) \\
\ll & \prod^{\mathcal{J}}_{j=1} \max \Big( \Big( 1+ O\big(e^{-\ell_j} \big) \Big)^{\frac {5-8k}{1-2k}}, 1 \Big )
\mathop{{\sum}^{h}}\limits_{f\in H_{\kappa}}   \prod^{\mathcal{J}}_{j=1} \Big ({\mathcal N}_j(f, 2k) +  {\mathcal Q}_j(f, 2k)  \Big )  \\
\ll & \mathop{{\sum}^{h}}\limits_{f\in H_{\kappa}}  \prod^{\mathcal{J}}_{j=1} \Big ( {\mathcal N}_j(f, 2k) +  {\mathcal Q}_j(f, 2k)   \Big ),
\end{split}
\end{align}
 where the last estimation above follows by noting that
\begin{align*}
\begin{split}
 \prod^{\mathcal{J}}_{j=1} \max \Big( \Big( 1+ O\big(e^{-\ell_j} \big) \Big)^{\frac {5-8k}{1-2k}}, 1 \Big ) \ll 1.
\end{split}
\end{align*}

  Similarly, we apply \eqref{Njest1} to see that when $k>1/2$ and $|{\mathcal P}_j(f)| \le \ell_j/(20(2k+1))$,
\begin{align}
\label{est11}
{\mathcal N}_j(f, 2k-1)^{\frac {2k}{2k-1}}
&= \exp( 2k {\mathcal P}_j(f))\Big( 1+ O\big( e^{-\ell_j} \big) \Big)^{\frac {2k}{2k-1}} = {\mathcal N}_j(f, 2k)
\Big( 1+ O\big(e^{-\ell_j} \big) \Big)^{\frac {1}{2k-1}} .
\end{align}

  Moreover, we note that when $k>1/2$ and $|{\mathcal P}_j(f)| > \ell_j/(20(2k+1))$,
\begin{align*}
%%\label{4.21}
\begin{split}
|{\mathcal N}_j(f, 2k-1)| &\le \sum_{r=0}^{\ell_j} \frac{|(2k-1){\mathcal P}_j(f)|^r}{r!} \le
|{\mathcal P}_j(f)|^{\ell_j} \sum_{r=0}^{\ell_j} \Big( \frac{(20(2k+1))}{\ell_j}\Big)^{\ell_j-r} \frac{(2k-1)^r}{r!}
\le \Big( \frac{24(2k+1) |{\mathcal P}_j(f)|}{\ell_j}\Big)^{\ell_j} .
\end{split}
\end{align*}
  It follows from this and the definition of ${\mathcal Q}_j(f, 2k)$
given in \eqref{defP} that when $k>1/2$ and $|{\mathcal P}_j(f)| > \ell_j/(20(2k+1))$,
\begin{align}
\label{est21}
{\mathcal N}_j(f, 2k-1)^{\frac {2k}{2k-1}}
&\leq  {\mathcal Q}_j(f, 2k).
\end{align}

  We deduce from \eqref{est11} and \eqref{est21} and argue as in the derivation of \eqref{upperboundprodofN} that when $k>1/2$, we have
\begin{align}
\label{upperboundprodofN1}
\begin{split}
  \mathop{{\sum}^{h}}\limits_{f\in H_{\kappa}}  \mathcal{N}(f, 2k-1)^{2k/(2k-1)} \ll & \mathop{{\sum}^{h}}\limits_{f\in H_{\kappa}}
  \prod^{\mathcal{J}}_{j=1} \Big ( {\mathcal N}_j(f, 2k) +  {\mathcal Q}_j(f, 2k)   \Big ).
\end{split}
\end{align}

   The assertions of the lemma now follow by substituting the estimations \eqref{upperboundprodofN} (resp. \eqref{upperboundprodofN1}) into
   \eqref{basiclowerbound} and \eqref{upperboundk=0} (resp. \eqref{basicboundkbig}). This completes the proof of the lemma.
\end{proof}

 Notice that we have $\mathop{{\sum}^{h}}\limits_{\substack{f \in H_{\kappa}}}1 \ll 1$ by Lemma \ref{LemDeltaest}.
It follows from this and Lemma \ref{lem1} that in order to achieve Theorem \ref{thmlowerbound}, it suffices to establish the following three propositions.
\begin{proposition}
\label{Prop4} With the notation as above, we have for all $k \geq 0$,
\begin{align*}
%%\label{Aestmation}
\mathop{{\sum}^{h}}\limits_{f\in H_{\kappa}} L(\frac 12, \operatorname{sym}^2 f)  \mathcal{N}(f, 2k-1) \gg  (\log \kappa)^{ \frac {(2k)^2+1}{2}}.
\end{align*}
\end{proposition}

\begin{proposition}
\label{Prop5} With the notation as above, we have for all $k \geq 0$,
\begin{align*}
%%\label{Nestmation}
 \mathop{{\sum}^{h}}\limits_{f\in H_{\kappa}} \prod^{\mathcal{J}}_{j=1} \Big ( {\mathcal N}_j(f, 2k) +  {\mathcal Q}_j(f, 2k)   \Big )\ll (\log \kappa)^{\frac {(2k)^2}{2}}.
\end{align*}
\end{proposition}

\begin{proposition}
\label{Prop6} With the notation as above and assuming the truth of GRH, we have for $0 \leq k < 1/2$,
\begin{align*}
%%\label{L2estmation}
\mathop{{\sum}^{h}}\limits_{f\in H_{\kappa}} |L(\frac 12, \operatorname{sym}^2 f)|^2{\mathcal N}(f, 2k-2) \ll ( \log \kappa )^{2k^2+1}.
\end{align*}
\end{proposition}

   In the remaining part of the paper, we shall prove the above propositions.

\section{Proof of Proposition \ref{Prop4}}
\label{sec 4}

    Define $\widetilde{\lambda}(n)$ to be the completely multiplicative function such that $\widetilde{\lambda}(p)=\lambda_f(p^2)$ on primes $p$ and $w(n)$ to be the multiplicative function such that
    $w(p^{\alpha}) = \alpha!$ for prime powers $p^{\alpha}$.
   Denote $\Omega(n)$ for the number of prime powers dividing $n$ and define functions $b_j(n), 1 \leq j \leq \mathcal{J}$ such that $b_j(n)=0$ or $1$ and $b_j(n)=1$ if and only if $\Omega(n) \leq \ell_j$ and the primes dividing $n$ are all from the interval $P_j$.
    We use these notations to write ${\mathcal N}_j(f, \alpha)$ as
\begin{equation}
\label{5.1}
{\mathcal N}_j(f, \alpha) = \sum_{n_j} \frac{\widetilde{\lambda}(n_j)}{\sqrt{n_j}} \frac{\alpha^{\Omega(n_j)}}{w(n_j)}  b_j(n_j) , \quad 1\le j\le \mathcal{J}.
\end{equation}
    Note that when $\kappa$ is large enough, we have
\begin{align}
\label{sumoverell}
 \mathcal{J} \ll \log \log \log \kappa, \quad \sum^{\mathcal{J}}_{j=1}\ell_j \leq 4e^A(\log \log \kappa)^{3/2}, \quad \sum^{\mathcal{J}}_{j=1} \alpha_{j}\ell_j \leq 40e^A 10^{-M/4}.
\end{align}
 It follows that ${\mathcal N}_j(f, \alpha)$ is a short Dirichlet polynomial since $b_j(n_j)=0$ unless $n_j \leq
    (\kappa^{\alpha_j})^{\ell_j}$. This implies that ${\mathcal N}(f, 2k-1)$ is also a short Dirichlet
    polynomial whose length is at most $\kappa^{\sum^{\mathcal{J}}_{j=1} \alpha_{j}\ell_j} < \kappa^{40e^A 10^{-M/4}}$ by \eqref{sumoverell}. We then write for simplicity that
\begin{align}
\label{Nexpression}
 {\mathcal N}(f, 2k-1)= \sum_{n  \leq \kappa^{40e^A 10^{-M/4}}} \frac{x_n}{\sqrt{n}} \widetilde{\lambda}(n),
\end{align}
    where we apply \eqref{sumoverell} to see that
\begin{align}
\label{xybounds}
 x_n \ll |2k-1|^{\sum^{\mathcal{J}}_{j=1}\ell_j} \ll \kappa^{\varepsilon}.
\end{align}

     We note that each $n$ appearing in the sum on the right-hand side expression of \eqref{Nexpression} can be written as
$n=\prod^{\mathcal{J}}_{j=1}n_j$ with $b_j(n_j)=1$ and this implies that $\widetilde{\lambda}(n)=\prod^{\mathcal{J}}_{j=1}\widetilde{\lambda}(n_j)$.
Using the relation \eqref{lambdamult}, we see
that $\lambda_f(p^2)^2$ can be written as a sum of $3$ terms of the form $\lambda_f(t^2)$ with $t | p^2$ with coefficients being either $0$ or  $1$. It follows from this and the definition of
$\widetilde{\lambda}(n_j)$ that each $\widetilde{\lambda}(n_j)$ can be written as a sum of at most $3^{\Omega(n_j)}$ (not necessary distinct) terms of the form $\lambda_f(t^2)$ with $t | n_j^2$ such that the coefficient of each term equals either $0$ or $1$. This further implies that $\widetilde{\lambda}(n)$ can be written as a sum of at most
 $3^{\sum^{\mathcal{J}}_{j=1}\Omega(n_j)}$ (not necessary distinct) terms of the form $\lambda_f(t^2)$ with $t | n^2$ such that the coefficient of each term equals either $0$ or $1$.  We group the identical terms together to write for simplicity that
\begin{align}
\label{lambdaf}
 \widetilde{\lambda}(n)=\sum_{t|n}c_n(t)\lambda_f(t^2),
\end{align}
 where $c_n(t) \geq 0$ and we have by \eqref{sumoverell},
\begin{align}
\label{cnt}
 \sum_{t|n}|c_n(t)| \leq 3^{\sum^{\mathcal{J}}_{j=1}\Omega(n_j)} \leq 3^{\sum^{\mathcal{J}}_{j=1}\ell_j}  \ll \kappa^{\varepsilon}.
\end{align}

  As  $\widetilde{\lambda}(n)$ is completely multiplicative, we have that $\widetilde{\lambda}(n_1n_2)=\widetilde{\lambda}(n_1)\widetilde{\lambda}(n_2)$. It
follows from this, the fact that $\lambda_f(n)$ is multiplicative and the relation given in \eqref{lambdaf} that for any $(n_1, n_2)=1$ and any $t|n_1n_2$, we may write $t$ uniquely as
$t=t_1t_2$ with $t_1|n_1, t_2|n_2$ such that
\begin{align}
\label{cmultiplicative}
  c_{n_1n_2}(t)= c_{n_1}(t_1)c_{n_2}(t_2).
\end{align}

  We now apply \eqref{Nexpression} and \eqref{lambdaf} to see that
\begin{align}
\label{Lfirstmoment}
\begin{split}
& \mathop{{\sum}^{h}}\limits_{f\in H_{\kappa}} L(\frac 12, \operatorname{sym}^2 f)  \mathcal{N}(f, 2k-1)
=   \sum_{n  \leq \kappa^{40e^A 10^{-M/4}}} \frac{x_n}{\sqrt{n}}\sum_{t|n}c_n(t)
\mathop{{\sum}^{h}}\limits_{f\in H_{\kappa}} L(\frac 12, \operatorname{sym}^2 f)\lambda_f(t^2).
\end{split}
\end{align}

  Note that the largest $t^2$ appearing on the right-hand side expression above does not exceed
$n^2 \ll \kappa$. We then apply \eqref{lstapproxk} to evaluate the inner sum above to see that, upon choosing $M$ large enough, the error term in \eqref{lstapproxk} contributes
\begin{align}
\label{sumchistar}
 \ll & \sum_{n  \leq \kappa^{40e^A 10^{-M/4}}} \frac{x_n}{\sqrt{n}}\sum_{t|n}\big| c_n(t)\big|(t^2\kappa^{-1/2+\varepsilon}+k^{-1}) \ll
 \kappa^{-1/4},
\end{align}
 where the last estimation above follows from \eqref{xybounds} and \eqref{cnt}.

 It therefore remains to consider the contribution to \eqref{Lfirstmoment} from the main terms in \eqref{lstapproxk}. Without loss of generality,
 we consider the contribution from the first term $-\log l/\sqrt{l}$ appearing on the right-hand side expression of \eqref{lstapproxk}. Writing $t=p^{l_0}_0l$
 with $(p_0,l)=1$ and $p_0$ being a prime, we see that this contribution is
\begin{align}
\label{sumoverlog}
\begin{split}
 -\sum_{\substack{ p_0 \in \bigcup^{\mathcal{J}}_{j=1} P_j }}\sum_{l_0 \geq 1}
\frac {l_0 \log p_0}{p^{l_0/2}_0}\sum_{n  \leq \kappa^{40e^A 10^{-M/4}}} \frac{x_n}{\sqrt{n}}  \sum_{\substack{ (l, p_0)=1 \\ p^{l_0}_0l|n}}\frac{c_n(p^{l_0}_0l)}{\sqrt{l}} .
\end{split}
\end{align}

   We consider the sum over a fixed $p_0 \in P_1$ in the above expression. Using the multiplicative relation \eqref{cmultiplicative}, we recast this sum as
\begin{align*}
%%\label{sumoverlog1}
\begin{split}
 -\sum_{l_0 \geq 1}
\frac {l_0 \log p_0}{p_0^{l_0/2}}\Big (\sum_{n_1} \frac{1}{\sqrt{n_1}} \frac{(2k-1)^{\Omega(n_1)}}{w(n_1)}  b_1(n_1)
\sum_{\substack{ (l_1, p_0)=1 \\ p_0^{l_0}l_1|n_1 }} \frac{c_{n_1}(p_0^{l_0}l_1)}{\sqrt{l_1}} \Big )\prod^{\mathcal{J}}_{j=2}
\Big (\sum_{n_j} \frac{1}{\sqrt{n_j}} \frac{(2k-1)^{\Omega(n_j)}}{w(n_j)}  b_j(n_j)
\sum_{\substack{ l_j|n_j }} \frac{c_{n_j}(l_j)}{\sqrt{l_j}} \Big ) .
\end{split}
\end{align*}

  Noticing that in the sum over $n_1$ above, we have $p_0^{l_0}|n_1$. Thus, upon making a change of variable $n_1 \mapsto p_0^{l_0}n_1$,
  we may recast the expression above as
\begin{align}
\label{sumovern1}
\begin{split}
-\sum_{l_0 \geq 1}
\frac {l_0 \log p_0}{p_0^{l_0}}\Big (\sum_{n_1} \frac{1}{\sqrt{n_1}} \frac{(2k-1)^{\Omega(p_0^{l_0}n_1)}}{w(p_0^{l_0}n_1)}  b_1(p_0^{l_0}n_1)
\sum_{\substack{ (l_1, p_0)=1 \\ l_1|n_1 }} \frac{c_{p_0^{l_0}n_1}(p_0^{l_0}l_1)}{\sqrt{l_1}} \Big )\prod^{\mathcal{J}}_{j=2}
\Big (\sum_{n_j} \frac{1}{\sqrt{n_j}} \frac{(2k-1)^{\Omega(n_j)}}{w(n_j)}  b_j(n_j)
\sum_{\substack{ l_j|n_j }} \frac{c_{n_j}(l_j)}{\sqrt{l_j}} \Big ) .
\end{split}
\end{align}

  We consider the sum above over $n_1$ above. Note that the factor $b_1(p_0^{l_0}n_1)$ restricts $p_0^{l_0}n_1$ to have
  all prime factors in $P_1$ such that $\Omega(p_0^{l_0}n_1) \leq \ell_1$. If we remove the restriction on $\Omega(p_0^{l_0}n_1)$, then the sum becomes
\begin{align*}
%%\label{6.02}
\begin{split}
& \Big( \sum_{i=0}^{\infty} \frac{1}{p_0^{i/2}} \frac{(2k-1)^{i+l_0}c_{p_0^{i+l_0}}(p_0^{l_0})}{(i+l_0)!}  \Big) \prod_{\substack{p \in P_1 \\ (p, p_0)=1}} \Big( \sum_{i=0}^{\infty} \frac{1}{p^{i/2}} \frac{(2k-1)^{i}}{i!}\Big ( \sum_{m=0}^{i}
\frac{c_{p^i}(p^m)}{\sqrt{p^{m}}} \Big ) \Big) \\
= & \frac {(2k-1)^{l_0}}{l_0!} \Big( \sum_{i=0}^{\infty} \frac{1}{p_0^{i/2}} \frac{(2k-1)^{i}c_{p_0^{i+l_0}}(p_0^{l_0})}{\binom {i+l_0}{l_0}i!}  \Big) \prod_{\substack{p \in P_1 \\ (p, p_0)=1}} \Big( \sum_{i=0}^{\infty} \frac{1}{p^{i/2}} \frac{(2k-1)^{i}}{i!}\Big ( \sum_{m=0}^{i}
\frac{c_{p^i}(p^m)}{\sqrt{p^{m}}} \Big ) \Big) \\
= & \frac {(2k-1)^{l_0}}{l_0!} \Big( c_{p_0^{l_0}}(p_0^{l_0})+\sum_{i=1}^{\infty} \frac{1}{p_0^{i/2}} \frac{(2k-1)^{i}c_{p_0^{i+l_0}}(p_0^{l_0})}{\binom {i+l_0}{l_0}i!}  \Big) \prod_{\substack{p \in P_1 \\ (p, p_0)=1}} \Big( \sum_{i=0}^{\infty} \frac{1}{p^{i/2}} \frac{(2k-1)^{i}}{i!}\Big ( \sum_{m=0}^{i}
\frac{c_{p^i}(p^m)}{\sqrt{p^{m}}} \Big ) \Big) \\
=& \frac {(2k-1)^{l_0}}{l_0!}\Big(1+O(\frac 1{p_0^{1/2}}) \Big )\prod_{\substack{p\in P_1 \\ (p, p_0)=1 }}\Big (1+ \frac {(2k)^2-1}{2p}+O(\frac 1{p^{3/2}}) \Big ),
\end{split}
\end{align*}
  where the last expression above follows from the observation that
\begin{align}
\label{cvalue}
\begin{split}
  c_p(1)=0, c_p(p)=c_{p^2}(1)=c_{p_0^{l_0}}(p_0^{l_0})=1,
\end{split}
\end{align}
 and the estimations that $\binom {i+l_0}{l_0} \geq 1, \sum^m_{i=0}|c_{p^i}(p^m)| \leq 3^{\Omega(p^i)}$.

   On the other hand, using Rankin's trick by noticing that $2^{\Omega(n_1)+l_0-\ell_1}\ge 1$ if $\Omega(p_0^{l_0}n_1) > \ell_1$,  we see via
Lemma \ref{RS} that the error introduced this way does not exceed
\begin{align*}
%%\label{6.021}
\begin{split}
&  2^{-\ell_1}\sum_{i=0}^{\infty} \frac{1}{p_0^{i/2}} \frac{|2(2k-1)|^{i+l_0}|c_{p_0^{i+l_0}}(p_0^{i+l_0})|}{(i+l_0)!}
\prod_{\substack{p \in P_1 \\ (p, p_0)=1}} \Big( \sum_{i=0}^{\infty} \frac{1}{p^{i/2}} \frac{|2(2k-1)|^{i}}{i!}\Big ( \sum_{m=0}^{i}
\frac{|c_{p^i}(p^m)|}{\sqrt{p^{m}}} \Big ) \Big) \\
\ll & 2^{-\ell_1/2}\frac {|2(2k-1)|^{l_0}}{l_0!}\Big(1+O(\frac 1{p_0^{1/2}}) \Big )\prod_{\substack{p\in P_1 \\ (p, p_0)=1 }}
\Big (1+ \frac {(2k)^2-1}{2p}+O(\frac 1{p^{3/2}}) \Big ),
\end{split}
\end{align*}
 where the last estimation above follows from \eqref{sumpj}.

   The above estimations carry out to the sums over other $n_j$ and we see this way that for any $2 \leq j \leq \mathcal{J}$, we have
\begin{align*}
%%\label{errorbound}
\begin{split}
 & \sum_{n_j} \frac{1}{\sqrt{n_j}} \frac{(2k-1)^{\Omega(n_j)}}{w(n_j)}  b_j(n_j)
\sum_{\substack{ l_j|n_j }} \frac{c_{n_j}(l_j)}{\sqrt{l_j}}  =  \Big (1+O(2^{-\ell_j/2}) \Big )\prod_{\substack{p\in P_j }}\Big (1+ \frac {(2k)^2-1}{2p}+O(\frac 1{p^{3/2}}) \Big ).
\end{split}
\end{align*}

  We then conclude from the above discussions that the expression given in \eqref{sumovern1} is
\begin{align*}
%%\label{Aestmation}
\leq & \sum_{l_0 \geq 1}
\frac {l_0 \log p_0}{p_0^{l_0}l_0!}|2k-1|^{l_0}(1+O(\frac 1{p_0^{1/2}}))\Big(1+O(2^{-\ell_1/2}2^{l_0}) \Big )\prod_{\substack{p\in P_1 \\ (p, p_0)=1 }}
\Big (1+ \frac {(2k)^2-1}{2p}+O(\frac 1{p^{3/2}}) \Big )\\
& \cdot \prod^{\mathcal{J}}_{j=2}\Big (1+O(2^{-\ell_j/2}) \Big )
\prod_{\substack{p\in P_j }}\Big (1+ \frac {(2k)^2-1}{2p}+O(\frac 1{p^{3/2}}) \Big ) \\
\ll & \Big (\frac {\log p_0}{p_0}+O(\frac {\log p_0}{p_0^{3/2}}) \Big ) \prod^{\mathcal{J}}_{j=1}
\prod_{\substack{p\in P_j }}\Big (1+ \frac {(2k)^2-1}{2p}+O(\frac 1{p^{3/2}}) \Big ) \\
\ll &  \Big (\frac {\log p_0}{p_0}+O(\frac {\log p_0}{p_0^{3/2}}) \Big )(\log \kappa)^{ \frac {(2k)^2-1}{2}},
\end{align*}
  where the last estimation above follows from the well-known relation that $1+x \leq e^x$ for all real number
$x$ and Lemma \ref{RS}.

  We sum over $p_0$ and apply Lemma \ref{RS} to see that the expression in \eqref{sumoverlog} is
\begin{align}
\label{Maintermcontri1}
\ll &  10^{-M/4}(\log \kappa)^{ \frac {(2k)^2+1}{2}}.
\end{align}

  On the other hand, the same procedure above implies that the contribution to \eqref{Lfirstmoment} from the other main terms in \eqref{lstapproxk} is
\begin{align}
\label{Maintermcontri2}
\gg & (\log \kappa)^{ \frac {(2k)^2+1}{2}}.
\end{align}

  The assertion of Proposition \ref{Prop4} now follows from \eqref{sumchistar}, \eqref{Maintermcontri1} and \eqref{Maintermcontri2}.

\section{Proof of Proposition \ref{Prop5}}

    For $1 \leq j \leq \mathcal{J}$, we define the function $p_{j}(n)$ such that $p_{j}(n)=0$ or $1$,
and that $p_{j}(n)=1$ if and only if $\Omega(n)=r_k\ell_{j}$ and all the prime factors of $n$ are from the interval $P_{j}$.
Using this together with the notations in Section \ref{sec 4} and recalling the definition of $c_k, r_k$ given in \eqref{r_k}, we see that
\begin{align*}
%%\label{Pexpression}
  {\mathcal Q}_{j}(f, 2k) =&  \Big( \frac{c_k  }{\ell_{j}}\Big)^{r_k\ell_{j}}\sum_{ \substack{ n_{j}}}
\frac{1}{\sqrt{n_{j}}}\frac{(r_k\ell_{j})!
  }{w(n_{j})}\widetilde{\lambda}(n_j) p_{j}(n_{j}).
\end{align*}

  Note that ${\mathcal
 Q}_{j}(f, 2k)$ is a short Dirichlet polynomial whose length does not exceed
$(\kappa^{\alpha_j})^{r_k\ell_j}=\kappa^{r_k\alpha_j\ell_j}$.
  Also, we have by \eqref{Stirling},
\begin{align}
\label{lv1est}
 \Big( \frac {c_k }{\ell_{j}}\Big)^{r_k\ell_{j}}(r_k\ell_{j})!  \ll (r_k\ell_{j}) \Big( \frac{c_k r_k }{e } \Big)^{r_k\ell_{j}}.
\end{align}

   We apply the above together with \eqref{Nexpression} and \eqref{xybounds} to see that there exists a constant $C(k)$ depending on $k$ only
such that for any $1 \leq j \leq \mathcal{J}$, we have for some $y_{n_j} \ll C(k)^{\ell_j}$ with $n_j \leq \kappa^{r_k\alpha_j\ell_j}$,
\begin{align*}
%%\label{Nexpression1}
 {\mathcal N}_j(f, 2k)+{\mathcal
 Q}_{j}(f, 2k)= \sum_{n_j} \frac{y_{n_j}}{\sqrt{n_j}}\widetilde{\lambda}(n_j).
\end{align*}

  We thus deduce that we may write for simplicity that
\begin{align}
\label{sumN}
\begin{split}
 \mathop{{\sum}^{h}}\limits_{f\in H_{\kappa}} \prod^{\mathcal{J}}_{j=1} \Big ( {\mathcal N}_j(f, 2k) +  {\mathcal Q}_j(f, 2k)   \Big ) =\mathop{{\sum}^{h}}\limits_{f\in H_{\kappa}}  \sum_{n}
\frac {y_n}{\sqrt{n}}\widetilde{\lambda}(n)=\sum_{n}
\frac {y_n}{\sqrt{n}}\sum_{t|n}c_{n}(t)\mathop{{\sum}^{h}}\limits_{f\in H_{\kappa}} \lambda_f(t^2),
\end{split}
\end{align}
  where $n  \leq \kappa^{r_k\sum^{\mathcal{J}}_{j=1}\alpha_j\ell_j} \leq \kappa^{40r_ke^A10^{-M/4}}$ by \eqref{sumoverell}. We also derive from \eqref{sumoverell} that
\begin{align*}
%%\label{ybounds}
  y_n \ll C(k)^{\sum^{\mathcal{J}}_{j=1}\ell_j}  \ll \kappa^{\varepsilon}.
\end{align*}

   We now apply \eqref{Deltaest} to evaluate the last sum on the right-hand side expression in \eqref{sumN} to see that, upon choosing $M$ large enough, the contribution from
the error term in \eqref{Deltaest} to the right-hand side expression in \eqref{sumN} is
\begin{align*}
%%\label{sumchistar}
 \ll & \sum_{n  \leq \kappa^{40r_ke^A10^{-M/4}}} \frac{y_n}{\sqrt{n}}\sum_{t|n}\big| c_n(t)\big|e^{-\kappa} \ll
 \sum_{n  \leq \kappa^{2r_k/10^{M/8}}} \kappa^{\varepsilon}\sum_{t|n}\big| c_n(t)\big|e^{-\kappa}  \ll \kappa^{-1/4},
\end{align*}
 where the last estimation above follows from \eqref{cnt}.

   It therefore remains to consider the contribution to the right-hand side expression in \eqref{sumN} from the main term in \eqref{Deltaest}, which equals to
\begin{align}
\label{sumNmain}
\begin{split}
 \sum_{n} \frac {y_n}{\sqrt{n}}c_{n}(1)=\prod^{\mathcal{J}}_{j=1}(\frac {y_{n_j}}{\sqrt{n_j}}c_{n_j}(1))=
\prod^{\mathcal{J}}_{j=1} \Big (  \sum_{n_j} \frac{(2k)^{\Omega(n_j)}}{\sqrt{n_j} w(n_j)}  b_j(n_j)c_{n_j}(1)
 + \Big( \frac{c_k  }{\ell_{j}}\Big)^{r_k\ell_{j}}\sum_{ \substack{ n_{j}}}
\frac{1}{\sqrt{n_{j}}}\frac{(r_k\ell_{j})!
  }{w(n_{j})}c_{n_j}(1) p_{j}(n_{j})\Big ).
\end{split}
\end{align}

  Arguing as in the proof of Proposition \ref{Prop4}, we see that
\begin{align}
\label{sqinN}
\begin{split}
  \sum_{n_j} \frac{(2k)^{\Omega(n_j)}}{\sqrt{n_j} w(n_j)}  b_j(n_j)c_{n_j}(1)=\Big(1+ O\big(2^{-\ell_j/2} \big ) \Big)
  \exp (\sum_{\substack{p \in P_j }} \frac {(2k)^2}{2p}+ O(\sum_{p \in P_j} \frac {1}{p^{3/2}})).
\end{split}
\end{align}

  Similarly, we notice that $(c_kr_k)^{\Omega(n_j)-r_k\ell_j}\ge 1$ when $\Omega(n_j) =r_k \ell_j$. It follows from this and \eqref{lv1est} that
\begin{align}
\label{Qest}
\begin{split}
 & \Big( \frac{c_k  }{\ell_{j}}\Big)^{r_k\ell_{j}}\sum_{ \substack{ n_{j}}}
\frac{1}{\sqrt{n_{j}}}\frac{(r_k\ell_{j})!
  }{w(n_{j})}c_{n_j}(1) p_{j}(n_{j})  \ll  \Big( \frac{c_k  }{\ell_{j}}\Big)^{r_k\ell_{j}}\sum_{ \substack{ n_{j}}}
\frac{(c_kr_k)^{\Omega(n_j)-r_k\ell_j}}{\sqrt{n_{j}}}\frac{(r_k\ell_{j})!
  }{w(n_{j})}|c_{n_j}(1)| \\
 \ll & r_k\ell_je^{-r_k\ell_j}
 \prod_{\substack{p \in P_j }}\Big (1+\frac {(c_kr_k)^2}{2p}+O(\frac 1{p^{3/2}}) \Big ),
\end{split}
\end{align}
 where the last estimation above follows from \eqref{cvalue} and the estimation that $|c_{p^n}(1)| \leq 3^n$.

 Further applying the estimation $1+x \leq e^x$, Lemma \ref{RS} and \eqref{sumpj}, we see that
\begin{align}
\label{Qest1}
\begin{split}
 & r_k\ell_je^{-r_k\ell_j}
 \sum_{\substack{p \in P_j }}\Big (1+\frac {(c_kr_k)^2}{2p}+O(\frac 1{p^{3/2}}) \Big )
\ll    r_k\ell_je^{-r_k\ell_j} \exp\Big ( \sum_{\substack{p \in P_j }}\frac {(c_kr_k)^2}{2p}+O(\sum_{p \in P_j} \frac {1}{p^{3/2}}) \Big ) \\
\ll &  2^{-\ell_j/2}  \exp (\sum_{\substack{p \in P_j }} \frac {(2k)^2}{2p}+ O(\sum_{p \in P_j} \frac {1}{p^{3/2}})).
\end{split}
\end{align}

  Substituting \eqref{sqinN}-\eqref{Qest1} into \eqref{sumNmain}, we see that
\begin{align*}
%%\label{sumNmain1}
\begin{split}
 \sum_{n} \frac {y_n}{\sqrt{n}}c_{n}(1)=\prod^{\mathcal{J}}_{j=1}\Big( \big (1+ O\big(2^{-\ell_j/2} \big ) \big )
  \exp (\sum_{\substack{p \in P_j }} \frac {(2k)^2}{2p}+ O(\sum_{p \in P_j} \frac {1}{p^{3/2}})) \Big) \ll (\log \kappa)^{\frac {(2k)^2}{2}}.
\end{split}
\end{align*}

  This completes the proof of Proposition \ref{Prop5}.

\section{Proof of Proposition \ref{Prop6}}
\label{sec 5}

\subsection{Initial treatment}
\label{sec 6.1}

  In the course of proving Proposition \ref{Prop6}, we need to first establish some weak upper bounds for moments of the family of symmetric square $L$-functions in this section, following the treatments in \cite{Sound2009}.  We denote
\begin{align*}
  \mathcal{N}(V)=\mathop{{\sum}^{h}}\limits_{\substack{f \in H_{\kappa} \\ \log\frac {|L(\frac 12, \operatorname{sym}^2 f)|}{\sqrt{\log \kappa}}\geq V }}1.
\end{align*}
 Our next result gives an upper bound for $\mathcal{N}(V)$.
\begin{prop}
\label{propNbound}
 With the notation as above and assuming the truth of GRH for $L(s, \operatorname{sym}^2 f)$. Let $k>0$ be a fixed real number.
  For $\log \log \kappa <V \leq   \frac{6\log \kappa}{\log \log \kappa}$, we have
 \begin{align*}
  \mathcal{N} (V)\ll (\log \kappa)^{9e^{16(2k+1)}+1}e^{-(2k+1)V}.
 \end{align*}
\end{prop}
\begin{proof}
  We apply \eqref{equ:3.3'} by setting $x=\kappa^{4/V}$ there and we denote $T$ for the sum in \eqref{equ:3.3'}. We then deduce that
\[
 \log\frac {|L(\frac 12, \operatorname{sym}^2 f)|}{\sqrt{\log \kappa}}
    \leq T+ (1+\lambda_0)\frac {V}{4} + O(\log \log \log \kappa).
\]
 As $\lambda_0<1$, this implies that if  $\log\frac {|L(\frac 12, \operatorname{sym}^2 f)|}{\sqrt{\log \kappa}} \geq V $, then we have
$T \geq \tfrac{V}{2}$.

Now, we define the harmonic measure of $T$ to be
\begin{align*}
\operatorname{meas}(T) & : = \mathop{{\sum}^{h}}\limits_{\substack{f \in H_{\kappa} \\ T \geq V/2}}1.
\end{align*}

  We keep the notations in Section \ref{sec 4} and we further define a totally multiplicative function $s(n, x)$ such that at primes $p$, we have
\begin{align*}
%%\label{spdef}
\begin{split}
  s(p,x)= \frac{1}{p^{\lambda_0/\log x}}\frac{\log (x/p)}{\log x}.
\end{split}
\end{align*}
  Note that we have $|s(n, x)| \leq 1$ for any $n$ whose prime factors do not exceed $x$. We then deduce by \eqref{lambdaf} that
\begin{align}
\label{M2bound}
\begin{split}
 &   \mathcal{N} (V) \leq \operatorname{meas}(T) \leq  (\frac {2}{V})^{2m}\mathop{{\sum}^{h}}\limits_{\substack{f \in H_{\kappa}}}|T|^{2m}
= (\frac {2}{V})^{2m}\mathop{{\sum}^{h}}\limits_{\substack{f \in H_{\kappa}}}\sum_{ \substack{ n \\ \Omega(n) = 2m \\ p|n \implies
p \leq x}}
\frac{(2m )!s(n,x)}{\sqrt{n}}\frac{1
  }{w(n)}\sum_{t|n}c_n(t)\lambda_f(t^2).
\end{split}
\end{align}

   Note that the largest $t^2$ appearing on the right-hand side expression above does not exceed
\begin{align*}
%%\label{Pexpression}
  n^2 \ll x^{4m}  \ll \kappa^{16m/V}.
\end{align*}

   We now take $m=\lceil V/16 \rceil-1$ to ensure that $16m/V<2$ so that we can apply \eqref{Deltaest} to evaluate the last sum in \eqref{M2bound} to see by the above discussions that the contribution from the error term in \eqref{Deltaest} is
\begin{align}
\label{MeasureMerror}
\begin{split}
  \ll &  e^{-\kappa}(\frac {2}{V})^{2m} \sum_{ \substack{ n \\ \Omega(n) = 2m \\ p|n \implies
p \leq x}}
\frac{(2m )!s(n,x)}{\sqrt{n}}\frac{1
  }{w(n)}\sum_{t|n}|c_n(t)|\\
\ll &  e^{-\kappa}(\frac {2}{V})^{2m} \big(\sum_{ p \leq x} \frac{3}{\sqrt{p}} \big )^{2m} \\
\ll &   e^{-\kappa}(3x^{1/2})^{2m} \\
\ll & \operatorname{exp}\left(-(2k+1)V\right).
\end{split}
\end{align}

  Meanwhile, using Rankin's trick that $(e^{8(2k+1)})^{\Omega(n)-2m } \geq 1$ when $\Omega(n)=2m$, the contribution from the main term in \eqref{Deltaest} is
\begin{align}
\label{MeasureMmain}
\begin{split}
   &  (\frac {2}{V})^{2m}\sum_{ \substack{ n \\ \Omega(n) = 2m \\ p|n \implies
p \leq x}}
\frac{(2m )!s(n,x)}{\sqrt{n}}\frac{1
  }{w(n)}c_n(1) \\
\ll &   (\frac {2}{V})^{2m} (2m )! \sum_{ \substack{ n \\  p|n \implies
p \leq x}}
\frac{(e^{8(2k+1)})^{\Omega(n)-2m }c_n(1)}{w(n)\sqrt{n}} \\
  \ll & (2m )\big( \frac {4m}{e^{8(2k+1)+1}V} \big )^{2m}\Big (\prod_{p \leq x } \big ( \sum^{\infty}_{i=0}\frac {e^{8(2k+1)i} c_{p^i}(1)}{i!p^{i/2}}\big)\Big ),
\end{split}
\end{align}
 where the last estimation above follows from \eqref{Stirling}  and the observation that $c_n(1)$ is multiplicative. Further using the observation $c_p(1)=0$ and the estimation that $|c_n(1)| \leq 3^{\Omega(n)}$ and that $1+x \leq e^x$, we see that
\begin{align*}
%%\label{Pmest}
\begin{split}
 \sum^{\infty}_{i=0}\frac {e^{8(2k+1)i}\cdot |c_{p^i}(1)|}{i!p^{i/2}} \ll & \exp (\frac {9e^{16(2k+1)}}{2p}).
\end{split}
\end{align*}
 We apply the above estimation and Lemma \ref{RS} to see that the last expression in  \eqref{MeasureMmain} is
\begin{align*}
%%\label{Pexpression}
 \ll V e^{-(2k+1)V}\exp (\prod_{p \leq x } \frac {9e^{16(2k+1)}}{2p} ) \ll V (\log \kappa)^{9e^{16(2k+1)}}e^{-(2k+1)V} \ll (\log \kappa)^{9e^{16(2k+1)}+1}e^{-(2k+1)V},
\end{align*}
 where the last estimation above follows by noticing that $V \leq \log \kappa$.

  The assertion of the proposition now follows from \eqref{MeasureMerror} and the above.
\end{proof}

  Now, Proposition \ref{propNbound} allows us to establish the following weak upper bounds for moments of $L$-functions concerned in the paper.
\begin{prop}
\label{prop: upperbound}
Assuming the truth of GRH for $L(f, \operatorname{sym}^2 f)$. For any positive real number $k$ and any $\varepsilon>0$, we have for large $\kappa$,
\begin{align*}
%%\label{upperbound1}
    \mathop{{\sum}^{h}}\limits_{f\in H_{\kappa}}|L(\frac 12, \operatorname{sym}^2 f)|^{2k}  \ll_k & (\log \kappa)^{O_k(1)}.
\end{align*}
\end{prop}
\begin{proof}
  Note that
\begin{align*}
%%\label{momentint}
  \mathop{{\sum}^{h}}\limits_{f\in H_{\kappa}}|L(\frac 12, \operatorname{sym}^2 f)|^{2k} =& -(\log \kappa)^{k} \int_{-\infty}^{+\infty}\operatorname{exp}(2kV)d \mathcal{N}(V)
  = 2k(\log \kappa)^{k} \int_{-\infty}^{+\infty}\operatorname{exp}(2kV)\mathcal{N}(V)dV.
\end{align*}
   As $N(V) \ll 1$ by \eqref{Deltaest}, we see that
\begin{align*}
 2k(\log \kappa)^{k} \int_{-\infty}^{10\log \log \kappa}\operatorname{exp}(2kV)\mathcal{N}(V)dV \ll (\log \kappa)^{k} \int_{-\infty}^{10\log \log \kappa}\operatorname{exp}(2kV) dV \ll (\log \kappa)^{O_k(1)}.
\end{align*}
  We may thus assume that $\log \log \kappa \leq V$ from now on.
   By taking $x = \log \kappa$ in \eqref{equ:3.3'} and bounding the sum over $p$ in \eqref{equ:3.3'} trivially, we see that $\mathcal{N} (V) =0$  for $V > \frac{6\log \kappa}{\log \log \kappa} $. Thus, we can also assume that $V \leq   \frac{6\log \kappa}{\log \log \kappa}$. We then apply Proposition \ref{propNbound} to see that the assertion of the proposition follows.
\end{proof}

\subsection{Completion of the proof}

    We keep the notations in Section \ref{sec 2'} and we set $x=\kappa^{\alpha_j}$ in \eqref{logLupperboundsimplified} to deduce that
\begin{align}
\label{basicest}
\begin{split}
 & \log |L(\frac 12, \operatorname{sym}^2 f)| \le \sum^{j}_{l=1} {\mathcal M}_{l,j}(f)+\sum_{0 \leq m \leq \frac{\log\log \kappa}{\log 2}} P_{m}(f)+\frac 12 \log \log \kappa+ 2\alpha^{-1}_j+O(1),
\end{split}
\end{align}
  where we set
\[ {\mathcal M}_{l,j}(f) = \sum_{p\in P_l}\frac{\lambda_f(p^2)}{\sqrt{p}}s(p, \kappa^{\alpha_j}), \quad 1\leq l \leq j \leq \mathcal{J} , \]
and
\[ P_{m}(f )=  \sum_{2^{m} < p \leq 2^{m+1}} \frac{\alpha^{4}_{f}(p, 1)+\alpha^{4}_{f}(p, 2)}{2p} , \quad  0 \leq m \leq \frac{\log\log \kappa}{\log 2}. \]

  We also define the following sets:
\begin{align*}
  \mathcal{S}(0) =& \{ f \in H_{\kappa} : |{\mathcal M}_{1,l}(f)| > \frac {\ell_{1}}{10^3} \; \text{ for some } 1 \leq l \leq \mathcal{J} \} ,   \\
 \mathcal{S}(j) =& \{ f \in H_{\kappa} : |{\mathcal M}_{m,l}(f)| \leq
 \frac {\ell_{m}}{10^3}  \; \forall 1 \leq m \leq j, \; \forall m \leq l \leq \mathcal{J}, \\
 & \;\;\;\;\; \text{but }  |{\mathcal M}_{j+1,l}(f)| > \frac {\ell_{j+1}}{200} \; \text{ for some } j+1 \leq l \leq \mathcal{J} \} ,  \quad  1\leq j \leq \mathcal{J}, \\
 \mathcal{S}(\mathcal{J}) =& \{f \in H_{\kappa} : |{\mathcal M}_{m,
\mathcal{J}}(f)| \leq \frac {\ell_{m}}{10^3} \; \forall 1 \leq m \leq \mathcal{J}\}, \\
\mathcal{P}(m) =&  \left\{ f \in H_{\kappa}  :  |P_{m}(f)| > 2^{-m/10} , \; \text{but} \; |P_{n}(f)| \leq 2^{-n/10} \; \mbox{for all} \; m+1 \leq n \leq \frac{\log\log \kappa}{\log 2} \right\}.
\end{align*}

  Note that we have
$|P_{n}(f)| \leq 2^{-n/10}$ for all $n$ if $f \not \in \mathcal{P}(m)$ for any $m$, so that we have by \eqref{alpha},
\[ \sum_{\substack{  p \leq \log \kappa }} \frac{\alpha^{4}_{f}(p, 1)+\alpha^{4}_{f}(p, 2)}{2p} = O(1). \]
 As the treatment for the case $f \not \in \mathcal{P}(m)$ for any $m$ is easier compared to the other cases, we may assume that $f \in \mathcal{P}(m)$ for some $m$.   We then observe that
\begin{align*}
%%\label{upperboundPm}
 \mathop{{\sum}^{h}}\limits_{\substack{f \in \mathcal{P}(m) }} 1 \leq  \mathop{{\sum}^{h}}\limits_{\substack{f \in H_{\kappa}}}
\Big (2^{m/10} | P_m(f)| \Big )^{2\lceil 2^{m/2}\rceil }.
\end{align*}

   Note that by \eqref{alpha} and \eqref{lambdafpnu}, we have
\begin{align*}
 \alpha^{4}_{f}(p, 1)+\alpha^{4}_{f}(p, 2)=\lambda_f(p^4)-\lambda_f(p^2).
\end{align*}
We define $\eta(n)$ to be the completely multiplicative function such that $\eta(p)=\lambda_f(p^4)-\lambda_f(p^2)$ on primes $p$. Similar to our discussion on Section \ref{sec 4}, we see that
\begin{align}
\label{eta}
 \eta(n)=\sum_{t|n^2}d_n(t)\lambda_f(t^2),
\end{align}
 where $d_n(t) \in \mr$ satisfying
\begin{align}
\label{dnt}
 \sum_{t|n^2}|d_n(t)| \leq 8^{\Omega(n)}.
\end{align}
It follows that we have
\begin{align}
\label{Pexpression}
\begin{split}
 \Big (2^{m/10} | P_m(f)| \Big )^{2\lceil 2^{m/2}\rceil } =&  2^{m\lceil 2^{m/2}\rceil/5 }\sum_{ \substack{ n \\ \Omega(n) = 2\lceil 2^{m/2}\rceil \\ p|n \implies
2^m <p \leq 2^{m+1}}}
\frac{1}{n}\frac{(2\lceil 2^{m/2}\rceil )!
  }{2^{\Omega(n)}w(n)}\eta(n)\\
= & 2^{m\lceil 2^{m/2}\rceil/5 }\sum_{ \substack{ n \\ \Omega(n) = 2\lceil 2^{m/2}\rceil \\ p|n \implies
2^m <p \leq 2^{m+1}}}
\frac{1}{n}\frac{(2\lceil 2^{m/2}\rceil )!
  }{2^{\Omega(n)}w(n)}\sum_{t|n^2}d_n(t)\lambda_f(t^2).
\end{split}
\end{align}

   Note that the largest $t^2$ appearing on the right-hand side expression above does not exceed
\begin{align*}
%%\label{Pexpression}
  n^4 \ll (2^{m+1})^{8\lceil 2^{m/2}\rceil} \ll \kappa.
\end{align*}

  We now apply \eqref{Deltaest} to evaluate the last sum in \eqref{Pexpression} to see by the above discussions that the contribution from the error term in \eqref{Deltaest} is
\begin{align*}
%%\label{Pmest}
\begin{split}
  \ll &  e^{-\kappa}2^{m\lceil 2^{m/2}\rceil/5 } \sum_{ \substack{ n \\ \Omega(n) = 2\lceil 2^{m/2}\rceil \\ p|n \implies
2^m <p \leq 2^{m+1}}}
\frac{(2\lceil 2^{m/2}\rceil )!}{n}\frac{1
  }{2^{\Omega(n)}w(n)}\sum_{t|n^2}|d_n(t)| \\
\ll &  e^{-\kappa}2^{m\lceil 2^{m/2}\rceil/5 } \big(\sum_{2^{m} < p \leq 2^{m+1}} \frac{4}{p} \big )^{2\lceil 2^{m/2}\rceil} \\
\ll & 2^{-2^{m/2}}.
\end{split}
\end{align*}

  Meanwhile, the contribution from the main term in \eqref{Deltaest} is
\begin{align*}
%%\label{Pmest}
\begin{split}
  \ll &  2^{m\lceil 2^{m/2}\rceil/5 } \sum_{ \substack{ n \\ \Omega(n) = 2\lceil 2^{m/2}\rceil \\ p|n \implies
2^m <p \leq 2^{m+1}}}
\frac{(2\lceil 2^{m/2}\rceil )!}{n}\frac{1
  }{2^{\Omega(n)}w(n)}|d_n(1)| \\
\ll &  2^{m\lceil 2^{m/2}\rceil/5 } (2\lceil 2^{m/2}\rceil )! \sum_{ \substack{ n \\  p|n \implies
2^m <p \leq 2^{m+1}}}
\frac{(2^{7m/10})^{\Omega(n)-2\lceil 2^{m/2}\rceil }|d_n(1)|}{n}\frac{1
  }{2^{\Omega(n)}w(n)} \\
  \ll & (2\lceil 2^{m/2}\rceil )\big( \frac {2\lceil 2^{m/2}\rceil }{2^{7m/10} e} \big )^{2\lceil 2^{m/2}\rceil} (2^{m/5})^{\lceil 2^{m/2}\rceil }\Big (\prod_{2^m < p <2^{m+1} } \big ( \sum^{\infty}_{i=0}\frac {2^{7mi/10}|d_{p^i}(1)|}{i!(2p)^i}\big)\Big ),
\end{split}
\end{align*}
 where the last estimation above follows from \eqref{Stirling}  and the observation that $d_n(1)$ is multiplicative. Further using the estimation that $1+x \leq e^x$ and $d_p(1)=0, |d_n(1)| \leq 8^{\Omega(n)}$, we see that the last expression above is
\begin{align}
\label{Pmest}
\begin{split}
\ll & (2\lceil 2^{m/2}\rceil )\big( \frac {2\lceil 2^{m/2}\rceil }{2^{7m/10} e} \big )^{2\lceil 2^{m/2}\rceil} (2^{m/5})^{\lceil 2^{m/2}\rceil }\exp \Big (\sum_{2^m < p <2^{m+1} } \frac {2^{7m/5}}{2p^2}+O(\frac {2^{21m/10}}{p^3}) \Big )\ll   2^{- 2^{m/2}}.
\end{split}
\end{align}

  We then apply H\"older's inequality to see that when $2^{m} \geq (\log\log \kappa)^{3}$,
\begin{align*}
 & \sum_{(\log\log \kappa)^3 \leq 2^m \leq \log \kappa} \mathop{{\sum}^{h}}\limits_{\substack{f \in \mathcal{P}(m) }}|L(\frac 12, \operatorname{sym}^2 f)|^{2} \mathcal{N}(f, 2k-2)  \\
\leq & \sum_{(\log\log \kappa)^3 \leq 2^m \leq \log \kappa} \Big ( \mathop{{\sum}^{h}}\limits_{\substack{f \in \mathcal{P}(m) }}1  \Big )^{1/4} \Big (
 \mathop{{\sum}^{h}}\limits_{\substack{f\in H_{\kappa}}}|L(\frac 12, \operatorname{sym}^2 f)|^{8}  \Big )^{1/4} \Big ( \mathop{{\sum}^{h}}\limits_{\substack{f\in H_{\kappa}}}\mathcal{N}(f, 2k-2)^{2}  \Big )^{1/2}.
\end{align*}

  Similar to the proof of Proposition \ref{Prop5}, we have that
\begin{align}
\label{N2k2bound}
&   \mathop{{\sum}^{h}}\limits_{\substack{f\in H_{\kappa}}} \mathcal{N}(f, 2k-2)^2  \ll ( \log \kappa  )^{O(1)}.
\end{align}

  Also, note that by Proposition \ref{prop: upperbound}, we have under GRH,
\begin{align}
\label{L8upperbound}
\mathop{{\sum}^{h}}\limits_{\substack{f\in H_{\kappa}}}|L(\frac 12, \operatorname{sym}^2 f)|^{8} \ll (\log \kappa)^{O(1)}.
\end{align}

  It follows that
\begin{align*}
 &  \sum_{(\log\log \kappa)^3 \leq 2^m \leq \log \kappa}\mathop{{\sum}^{h}}\limits_{\substack{f \in \mathcal{P}(m) }}|L(\frac 12, \operatorname{sym}^2 f)|^{2} \mathcal{N}(f, 2k-2) \\
 \ll & \sum_{(\log\log \kappa)^3 \leq 2^m \leq \log \kappa} \exp\left( -(\log 2)(\log\log \kappa)^{3/2}/4 \right)(\log \kappa)^{O_k(1)}  \\
 \ll & (\log \kappa)^{2k^2+1}(\log \log \kappa)^{-1}.
\end{align*}

   We deduce from the above that we may also assume that $0 \leq m \leq (3/\log 2)\log\log\log \kappa$. Further note that
\begin{align*}
%%\label{S0est}
\begin{split}
 \mathop{{\sum}^{h}}\limits_{\substack{f \in \mathcal{S}(0) }}1  \leq &  \mathop{{\sum}^{h}}\limits_{\substack{f\in H_{\kappa}}} \sum^{\mathcal{J}}_{l=1}
\Big ( \frac {10^3}{\ell_1}{|\mathcal
M}_{1, l}(f)| \Big)^{2\lceil 1/(10^3\alpha_{1})\rceil }.
\end{split}
\end{align*}

   Recall the function $s(n, x)$ defined in Section \ref{sec 6.1}, we see by \eqref{lambdaf} that the right-hand side expression above equals
\begin{align}
\label{M1fbound}
\begin{split}
 &  \mathop{{\sum}^{h}}\limits_{\substack{f\in H_{\kappa}}}  \sum^{\mathcal{J}}_{l=1}
\Big ( \frac {10^3}{\ell_1} {|\mathcal
M}_{1, l}(f)| \Big)^{2\lceil 1/(10^3\alpha_{1})\rceil } \\
= & \sum^{\mathcal{J}}_{l=1}
\mathop{{\sum}^{h}}\limits_{\substack{f\in H_{\kappa}}} \Big ( \frac {10^3}{\ell_1} \Big)^{2\lceil /(10^3\alpha_{1})\rceil } \sum_{ \substack{ n \\ \Omega(n) = 2\lceil 1/(10^3\alpha_{1})\rceil \\ p|n \implies
p \in P_1}}
\frac{(2\lceil 1/(10^3\alpha_{1})\rceil  )!s(n,\kappa^{\alpha_l})}{\sqrt{n}}\frac{1
  }{w(n)}\sum_{t|n}c_n(t)\lambda_f(t^2).
\end{split}
\end{align}

   Note that the largest $t^2$ appearing on the right-hand side expression above does not exceed
\begin{align*}
%%\label{Pexpression}
  n^2 \ll (\kappa^{\alpha_1})^{4\lceil 1/(10^3\alpha_{1})\rceil}  \ll \kappa.
\end{align*}

   We then apply \eqref{Deltaest} to evaluate the last sum in \eqref{M1fbound} to see by the above discussions together with the observation that $|s(n, \kappa^{\alpha_l})| \leq 1$ for all $n$ appearing in the sum and the estimations given in \eqref{Stirling}, \eqref{sumoverell}, \eqref{cnt} that the contribution from the error term in \eqref{Deltaest} is
\begin{align*}
%%\label{S0bounderror}
\begin{split}
  \ll &  e^{-\kappa}\mathcal{J} \Big ( \frac {10^3}{\ell_1} \Big)^{2\lceil 1/(10^3\alpha_{1})\rceil } \sum_{ \substack{ n \\ \Omega(n) = 2\lceil 1/(10^3\alpha_{1})\rceil \\ p|n \implies
p \in P_1}}
\frac{(2\lceil 1/(10^3\alpha_{1})\rceil  )!s(n,\kappa^{\alpha_l})}{\sqrt{n}}\frac{1
  }{w(n)}\sum_{t|n}c_n(t)\\
\ll &  e^{-\kappa}\mathcal{J}\Big ( \frac {10^3}{\ell_1} \Big)^{2\lceil 1/(10^3\alpha_{1})\rceil } \big(\sum_{p \in P_1} \frac{3}{\sqrt{p}} \big )^{2\lceil 1/(10^3\alpha_{1})\rceil} \\
\ll & \kappa e^{-(\log\log \kappa)^{2}/20} .
\end{split}
\end{align*}

   Meanwhile, the contribution from the main term in \eqref{Deltaest}  is
\begin{align}
\label{S0bound}
\begin{split}
  \ll & \mathcal{J} \Big ( \frac {10^3}{\ell_1} \Big)^{2\lceil 1/(10^3\alpha_{1})\rceil } \sum_{ \substack{ n \\ \Omega(n) = 2\lceil 1/(10^3\alpha_{1})\rceil \\ p|n \implies
p \in P_1}}
\frac{(2\lceil 1/(10^3\alpha_{1})\rceil  )!s(n,\kappa^{\alpha_l})}{\sqrt{n}}\frac{1
  }{w(n)}c_n(1) \\
\ll &  \mathcal{J} \Big ( \frac {10^3}{\ell_1} \Big)^{2\lceil 1/(10^3\alpha_{1})\rceil }(2\lceil 1/(10^3\alpha_{1})\rceil  )!
 \sum_{ \substack{ n \\ \Omega(n) = 2\lceil 1/(10^3\alpha_{1})\rceil \\ p|n \implies
p \in P_1}}
\frac{(\alpha^{-1/4}_{1})^{\Omega(n)-2\lceil 1/(10^3\alpha_{1})\rceil }c_n(1)}{w(n)\sqrt{n}} \\
  \ll & \mathcal{J} (2\lceil 1/(10^3\alpha_{1})\rceil)\big( \frac {2\lceil 1/(10^3\alpha_{1})\rceil \alpha^{1/4}_1}{ e} \big )^{2\lceil 1/(10^3\alpha_{1})\rceil} \Big ( \frac {10^3}{\ell_1} \Big)^{2\lceil 1/(10^3\alpha_{1})\rceil }\Big (\prod_{p \in P_1 } \big ( \sum^{\infty}_{i=0}\frac {\alpha^{-i/4}_1c_{p^i}(1)}{i!p^{i/2}}\big)\Big ),
\end{split}
\end{align}
 where the last estimation above follows from \eqref{Stirling}  and the observation that $c_n(1)$ is multiplicative. Further using the estimation that $|c_n(1)| \leq 3^{\Omega(n)}$, we see that when $p \leq (6\alpha^{-1/4}_1)^2$,
\begin{align*}
%%\label{Pmest}
\begin{split}
 \sum^{\infty}_{i=0}\frac {\alpha^{-i/4}_1c_{p^i}(1)}{i!p^{i/2}} \ll & \exp (\frac {3\alpha^{-1/4}_1}{\sqrt{p}}).
\end{split}
\end{align*}

  On the other hand, using the observation $c_p(1)=0$ and the estimation that $1+x \leq e^x$ we see that when $p > (6\alpha^{-1/4}_1)^2$, we have
\begin{align*}
%%\label{Pmest}
\begin{split}
 \sum^{\infty}_{i=0}\frac {\alpha^{-i/4}_1c_{p^i}(1)}{i!p^{i/2}} \leq & \exp (\frac {\alpha^{-1/2}_1}{2p}+\frac {9\alpha^{-3/4}_1}{p^{3/2}}).
\end{split}
\end{align*}

  We apply the above two estimations in the last expression in \eqref{S0bound} to see that it is
\begin{align}
\label{S0main}
\begin{split}
\ll & \mathcal{J} (2\lceil 1/(10^3\alpha_{1})\rceil)\big( \frac {2\lceil 1/(10^3\alpha_{1})\rceil \alpha^{1/4}_1}{ e} \big )^{2\lceil 1/(10^3\alpha_{1})\rceil} \Big ( \frac {10^3}{\ell_1} \Big)^{2\lceil 1/(10^3\alpha_{1})\rceil } \\
& \times \exp (\sum_{p  \leq (6\alpha^{-1/4}_1)^2} \frac {3\alpha^{-1/4}_1}{\sqrt{p}})\exp (\sum_{p  \leq \kappa^{\alpha_1}} (\frac {\alpha^{-1/2}_1}{2p}+\frac {9\alpha^{-3/4}_1}{p^{3/2}})).
\end{split}
\end{align}

 We now apply Lemma \ref{RS} and the definition of $\alpha_1$ given in \eqref{alphadef} to see that
\begin{align*}
%%\label{S0main}
\begin{split}
 \sum_{p  \leq \kappa^{\alpha_1}} (\frac {\alpha^{-1/2}_1}{2p}+\frac {9\alpha^{-3/4}_1}{p^{3/2}})= (\frac 12+o(1))\alpha_1^{-1}.
\end{split}
\end{align*}

  It follows from the above estimation that upon taking $A$ large enough in the definition of $\ell_j$,
  the last expression in \eqref{S0main} implies that
 \begin{align}
\label{S0upperbound}
\begin{split}
 \mathop{{\sum}^{h}}\limits_{\substack{f \in \mathcal{S}(0) }}1 \ll  e^{-\alpha_1^{-1}/20}=e^{-(\log\log \kappa)^{2}/20}  .
\end{split}
\end{align}

   We then deduce via H\"older's inequality that
\begin{align}
\label{LS0bound}
\begin{split}
& \mathop{{\sum}^{h}}\limits_{\substack{f \in S(0)}}|L(\frac 12, \operatorname{sym}^2 f)|^{2} \mathcal{N}(f, 2k-2)
\leq  \Big ( \mathop{{\sum}^{h}}\limits_{\substack{f \in S(0)}} 1\Big )^{1/4} \Big (
\mathop{{\sum}^{h}}\limits_{\substack{f \in \in H_{\kappa}}} |L(\frac 12, \operatorname{sym}^2 f)|^{8}  \Big )^{1/4} \Big ( \mathop{{\sum}^{h}}\limits_{\substack{f \in \in H_{\kappa}}} \mathcal{N}(f, 2k-2)^{2}   \Big )^{1/2}.
\end{split}
\end{align}

  We use the bounds given in \eqref{N2k2bound} and \eqref{L8upperbound} in \eqref{LS0bound} to conclude that
\begin{align*}
%%\label{S0bound1}
\mathop{{\sum}^{h}}\limits_{\substack{f\in S(0)}}|L(\frac 12, \operatorname{sym}^2 f)|^{2} \mathcal{N}(f, 2k-2)    \ll (\log \kappa)^{2k^2+1}.
\end{align*}

  We now define
\begin{align*}
  \mathcal{T} =& \{ f \in H_{\kappa} : |{\mathcal P}_{1}(f)| \leq  \frac {\ell_1}{10^3} \}.
\end{align*}
   We denote $\mathcal{T}^c$ for the complementary of $\mathcal{T}$ in $H_{\kappa}$. Similar to our approach above, we have
\begin{align*}
%%\label{P1est}
\begin{split}
 \mathop{{\sum}^{h}}\limits_{\substack{f \in \mathcal{T}  }}1  \leq &  \mathop{{\sum}^{h}}\limits_{\substack{f\in H_{\kappa} }}
\Big ( \frac {10^3}{\ell_1}{|\mathcal
P}_{1}(f)| \Big)^{2\lceil /(10^3\alpha_{1})\rceil }\ll  \kappa e^{-(\log\log \kappa)^{2}/20}.
\end{split}
\end{align*}

  Thus we may further assume that $j \geq 1$ and $f \in \mathcal{T}$. Note that
\begin{align*}
 & \left\{ f \in \mathcal{T}, f \in \mathcal{P}(m), f \in \mathcal{S}(j), 0 \leq m \leq (3/\log 2)\log\log\log \kappa, 1 \leq j \leq \mathcal{J} \right \} =  \bigcup_{m=0}^{(3/\log 2)\log\log\log \kappa}\bigcup_{j=1}^{ \mathcal{J}} \Big (\mathcal{S}(j)\bigcap \mathcal{P}(m) \bigcap \mathcal{T}\Big ).
\end{align*}
 It thus suffices to show that
\begin{align}
\label{sumovermj}
  \sum_{m=0}^{(3/\log 2)\log\log\log \kappa}\sum_{j=1}^{\mathcal{J}}\mathop{{\sum}^{h}}\limits_{\substack{f \in \mathcal{S}(j)\bigcap \mathcal{P}(m)\bigcap \mathcal{T}}} |L(\frac 12, \operatorname{sym}^2 f)|^{2} \mathcal{N}(f, 2k-2)
   \ll (\log \kappa)^{2k^2+1}.
\end{align}

  We consider the above sum of $|L(\frac 12, \operatorname{sym}^2 f)|^{2} \mathcal{N}(f, 2k-2)$ over $\mathcal{S}(j)\bigcap \mathcal{P}(m) \bigcap \mathcal{T}$ by fixing an $m$ such that $0 \leq m \leq (3/\log 2)\log\log\log \kappa$ and fixing a $j$ such that $1 \leq j \leq \mathcal{J}$.  We deduce from \eqref{basicest} that
\begin{align}
\label{zetaNbounds}
\begin{split}
 &  |L(\frac 12, \operatorname{sym}^2 f)|^{2} \mathcal{N}(f, 2k-2) \\
\ll & (\log \kappa)\exp \left(\frac {4}{\alpha_j} \right) \exp \Big (
 2\sum^j_{l=1}{\mathcal M}_{l,j}(f)+2 \sum^{\log \log \kappa /2}_{m=0}P_m(f) \Big )\mathcal{N}(f, 2k-2) \\
\ll & (\log \kappa)\exp \left(\frac {4}{\alpha_j} \right) \exp \Big (
 2\sum^j_{l=1}{\mathcal M}_{l,j}(f)+2 \sum^{\log \log \kappa /2}_{m=0}P_m(f) \Big )\mathcal{N}_1(f, 2k-2)\prod^{\mathcal J}_{l=2}\mathcal{N}_l(f, 2k-2).
\end{split}
\end{align}

   As $f \in \mathcal{T}$, we set $\alpha= 2k-2$ in \eqref{Njest1} to deduce that
\begin{align}
\label{N1bound}
\begin{split}
  \mathcal{N}_1(f, 2k-2)= \exp\Big ( (2k-2) {\mathcal P}_{1}(f)\Big ) (1+O(e^{-\ell_1})).
\end{split}
 \end{align}

  We apply the above in \eqref{zetaNbounds} to see that
\begin{align*}
%%\label{zetaNbounds1}
\begin{split}
 &  |L(\frac 12, \operatorname{sym}^2 f)|^{2} \mathcal{N}(f, 2k-2) \\
\ll & (\log \kappa)\exp \left(\frac {4}{\alpha_j} \right) \exp \Big (
 2 {\mathcal M}_{1,j}(f)+ 2(k-1) {\mathcal P}_{1}(f)+2 \sum^j_{l=2}{\mathcal M}_{l,j}(f)+2 \sum^{\log \log \kappa /2}_{m=0}P_m(f) \Big ) \\
&\times \prod^{\mathcal J}_{l=2}\mathcal{N}_l(f, 2k-2) .
\end{split}
\end{align*}

   We now want to separate the sums over $p \leq 2^{m+1}$ on the right-hand side expression above from those over $p >2^{m+1}$. To do so, we note that if $f \in \mathcal{P}(m)$, then
\begin{align}
\label{sump}
\begin{split}
 & \Big | 2 \sum_{  p  \leq 2^{m+1}}  \frac{\lambda_f(p^{2})}{\sqrt{p}}\frac{1}{p^{1/(\alpha_j\log \kappa)}}\frac{\log (\kappa^{\alpha_j}/p)}{\log \kappa^{\alpha_j}}+2(k-1) \sum_{  p  \leq 2^{m+1}} \frac{\lambda_f(p^{2})}{\sqrt{p}}+
  2 \sum_{p  \leq \log \kappa} \frac{\alpha^{4}_{f}(p, 1)+\alpha^{4}_{f}(p, 2)}{2p} \Big |
   \\
 \leq &  2\Big |\sum_{  p  \leq 2^{m+1}}  \frac{\lambda_f(p^{2})}{\sqrt{p}}\frac{1}{p^{1/(\alpha_j\log \kappa)}}\frac{\log (\kappa^{\alpha_j}/p)}{\log \kappa^{\alpha_j}}\Big |+2(1-k)\Big |\sum_{  p  \leq 2^{m+1}} \frac{\lambda_f(p^{2})}{\sqrt{p}}\Big |+
  \Big |\sum_{p  \leq 2^{m+1}}\frac{\alpha^{4}_{f}(p, 1)+\alpha^{4}_{f}(p, 2)}{p}\Big |
  +O(1),
\end{split}
\end{align}

  Now, using the estimation $\lambda_f(p^2) \leq d(p^2)=3$ and the relation given in \eqref{alpha}, we apply Lemma \ref{RS} and partial summation to see that
\begin{align*}
%%\label{sump}
\begin{split}
 &  2\Big |\sum_{  p  \leq 2^{m+1}}  \frac{\lambda_f(p^{2})}{\sqrt{p}}\frac{1}{p^{1/(\alpha_j\log \kappa)}}\frac{\log (\kappa^{\alpha_j}/p)}{\log \kappa^{\alpha_j}}\Big |+2(1-k)\Big |\sum_{  p  \leq 2^{m+1}} \frac{\lambda_f(p^{2})}{\sqrt{p}}\Big | \ll 6(2-k)\sum_{  p  \leq 2^{m+1}} \frac{1}{\sqrt{p}} \ll_k \frac {2^{m/2}}{m}, \\
& \Big |\sum_{p  \leq 2^{m+1}} \frac{\alpha^{4}_{f}(p, 1)+\alpha^{4}_{f}(p, 2)}{p}\Big | \leq 2\sum_{p  \leq 2^{m+1}} \frac{1}{p}
   =2\log (m+1) \ll \frac {2^{m/2}}{m}.
\end{split}
\end{align*}
  We then conclude that there is a constant $E$ depending on $k$ only such that the expressions in \eqref{sump} are
  $\leq \frac {E2^{m/2}}{m}$. It follows that
\begin{align}
\label{LboundinSP0}
\begin{split}
  &  \mathop{{\sum}^{h}}\limits_{\substack{f \in \mathcal{S}(j)\bigcap \mathcal{P}(m)\bigcap \mathcal{T}}} |L(\frac 12, \operatorname{sym}^2 f)|^{2} \mathcal{N}(f, 2k-2) \\
   \ll & (\log \kappa) e^{\frac {E2^{m/2}}{m}} \exp \left(\frac {4}{\alpha_j} \right) \mathop{{\sum}^{h}}\limits_{\substack{f \in \mathcal{S}(j)\bigcap \mathcal{P}(m)}}
\exp \Big ( 2 {\mathcal M}'_{1,j}(f)+2 \sum^j_{l=2}{\mathcal M}_{l,j}(f)\Big )\prod^{\mathcal J}_{l=2}\mathcal{N}_l(f, 2k-2) \\
  \\
\ll &  (\log \kappa) e^{\frac {E2^{m/2}}{m}} \exp \left(\frac {4}{\alpha_j} \right) \\
& \times \mathop{{\sum}^{h}}\limits_{\substack{f \in \mathcal{S}(j)}} \Big (2^{m/10}|P_m(f)| \Big )^{2\lceil 2^{m/2}\rceil }
\exp \Big ( 2 {\mathcal M}'_{1,j}(f)+2 \sum^j_{l=2}{\mathcal M}_{l,j}(f)\Big )\prod^{\mathcal J}_{l=2}\mathcal{N}_l(f, 2k-2),
\end{split}
 \end{align}
   where we define
\begin{align*}
\begin{split}
 {\mathcal M}'_{1,j}(f)= \sum_{ 2^{m+1}< p \leq \kappa^{\alpha_1} }
 \frac{\lambda_f(p^{2})}{\sqrt{p}}c_j(p,k),
\end{split}
\end{align*}
  and where we set for any $1 \leq j \leq \mathcal{J}$ and any real number $n$,
\begin{align*}
%%\label{cpdef}
\begin{split}
  c_j(p,n)= s(p, \kappa^{\alpha_j})+n-1.
\end{split}
 \end{align*}

 As $0 \leq m \leq (3/\log 2)\log\log\log \kappa $ and $\kappa$ is large, we have by Lemma \ref{RS} and partial summation,
\begin{align*}
%%\label{psumest}
\begin{split}
 & \Big | \sum_{ p \leq 2^{m+1}  } \frac{\lambda_f(p^{2})}{\sqrt{p}}c_j(p,k)\Big |
\leq  3(2-k)\sum_{ p < 2^{m+1}  }
 \frac{1}{\sqrt{p}} \leq 100(2-k)(\log \log \kappa )^{3/2}(\log \log \log \kappa )^{-1}.
\end{split}
 \end{align*}

  It follows from this that for $f \in \mathcal{S}(j)\bigcap \mathcal{T} $ and large $\kappa$,
\begin{align*}
%%\label{M'est}
\begin{split}
  |2{\mathcal M}'_{1,j}(f)| \leq & 200(2-k)(\log \log \kappa )^{3/2}(\log \log \log \kappa )^{-1}+2|{\mathcal M}_{1,j}(f)|+2|(k-1){\mathcal P}_{1}(f)| \leq \frac {\ell_1}{60} .
\end{split}
 \end{align*}

  The above estimation allows us to apply \eqref{Njest1} to see that
\begin{align}
\label{eM1bound}
\begin{split}
\exp\Big ( 2{\mathcal M}'_{1,j}(f)\Big ) \ll  E_{\ell_1}( 2{\mathcal M}'_{1,j}(f)) .
\end{split}
 \end{align}

    As we also have $|{\mathcal M}_{l, j}(f)| \leq  \ell_l/10^3$ when $f \in \mathcal{S}(j)$, we repeat our arguments above to see that,
\begin{align}
\label{eMNprodest1}
\begin{split}
\exp \Big ( 2 {\mathcal M}_{l,j}(f)\Big ) \leq
(1+O(e^{-\ell_l})) E_{\ell_l}( 2{\mathcal M}_{l,j}(f)).
\end{split}
 \end{align}

   We apply the estimations given in \eqref{eM1bound}, \eqref{eMNprodest1} in \eqref{LboundinSP0} to see that
\begin{align*}
%%\label{LboundinSP1}
\begin{split}
  &  \mathop{{\sum}^{h}}\limits_{\substack{f \in \mathcal{S}(j)\bigcap \mathcal{P}(m)\bigcap \mathcal{T}}} |L(\frac 12, \operatorname{sym}^2 f)|^{2} \mathcal{N}(f, 2k-2) \\
\ll &  (\log \kappa) e^{\frac {E2^{m/2}}{m}} \exp \left(\frac {4}{\alpha_j} \right) \sum_{f \in \mathcal{S}(j)} \Big (2^{m/10}|P_m(f)| \Big )^{2\lceil 2^{m/2}\rceil } E_{\ell_1}( 2{\mathcal M}'_{1,j}(f))  \\
& \times
\prod^{j}_{l=2} \Big (1+O\big(e^{-\ell_l}\big)\Big )E_{\ell_l}( 2{\mathcal M}_{l,j}(f)){\mathcal N}_{l}(f, 2k-2)  \times  \prod^{\mathcal J}_{l=j+1} {\mathcal N}_{l}(f, 2k-2) \\
\ll &  (\log \kappa) e^{\frac {E2^{m/2}}{m}} \exp \left(\frac {4}{\alpha_j} \right) \mathop{{\sum}^{h}}\limits_{\substack{f \in \mathcal{S}(j)}} \Big (2^{m/10}|P_m(f)| \Big )^{2\lceil 2^{m/2}\rceil }E_{\ell_1}( 2{\mathcal M}'_{1,j}(f))  \\
& \times
\prod^{j}_{l=2} E_{\ell_l}( 2{\mathcal M}_{l,j}(f)){\mathcal N}_{l}(f, 2k-2)  \times  \prod^{\mathcal J}_{l=j+1} {\mathcal N}_{l}(f, 2k-2).
\end{split}
\end{align*}

   We then deduce from the description on $\mathcal{S}(j)$ and the above that when $j \geq 1$,
\begin{align}
\label{upperboundprodE0}
\begin{split}
& \mathop{{\sum}^{h}}\limits_{\substack{f \in \mathcal{S}(j)\bigcap \mathcal{P}(m)\bigcap \mathcal{T}}}|L(\frac 12, \operatorname{sym}^2 f)|^{2} \mathcal{N}(f, 2k-2) \\
\ll & (\log \kappa) e^{\frac {E2^{m/2}}{m}} \exp \left(\frac {4}{\alpha_j} \right)
 \sum^{ \mathcal{J}}_{u=j+1}\mathop{{\sum}^{h}}\limits_{\substack{f \in \mathcal{S}(j)}} \Big (2^{m/10}|P_m(f)| \Big )^{2\lceil 2^{m/2}\rceil }E_{\ell_1}( 2{\mathcal M}'_{1,j}(f))  \\
& \times \prod^{j}_{l=2} E_{\ell_l}( 2{\mathcal M}_{l,j}(f)){\mathcal N}_{l}(f, 2k-2)\prod^{\mathcal J}_{l=j+2} {\mathcal N}_{l}(f, 2k-2) \\
& \times \Big ( \frac {10^3}{\ell_{j+1}}|{\mathcal M}_{j+1,u}(f)|\Big)^{2\lceil 1/(10\alpha_{j+1})\rceil }{\mathcal N}_{j+1}(f, 2k-2).
\end{split}
\end{align}

   We further simplify the right-hand expression above by noting that when  $|(2k-2){\mathcal P}_{j+1}(f)| \leq \ell_{j+1}/60$, we have similar to \eqref{N1bound},
\begin{align}
\label{Njplusonebound0}
\begin{split}
 \mathcal{N}_{j+1}(f, 2k-2) \ll  \exp\Big ( 2(k-1) {\mathcal P}_{j+1}(f)\Big ) \ll \exp\Big ( \ell_{j+1}/60 \Big )
\ll 2^{2\lceil 1/(10^3\alpha_{j+1})\rceil }.
\end{split}
\end{align}
   While when $|(2k-2){\mathcal P}_{j+1}(f)| > \ell_{j+1}/60$, we have similar to \eqref{N22kbound},
\begin{align}
\label{Njplusonebound}
\begin{split}
 \mathcal{N}_{j+1}(f, 2k-2) \leq \Big( \frac{64 |{\mathcal P}_{j+1}(f)|}{\ell_{j+1}}\Big)^{\ell_{j+1}} \leq \Big( \frac {120 |{\mathcal P}_{j+1}(f)|}{\ell_{j+1}}\Big)^{2\lceil 1/(10^3\alpha_{j+1})\rceil }.
\end{split}
\end{align}

   We apply the estimations given in \eqref{Njplusonebound0} and \eqref{Njplusonebound} in \eqref{upperboundprodE0} to deduce that
\begin{align*}
%%\label{upperboundprodE1}
\begin{split}
& \mathop{{\sum}^{h}}\limits_{\substack{f \in \mathcal{S}(j)\bigcap \mathcal{P}(m)\bigcap \mathcal{T}}}|L(\frac 12, \operatorname{sym}^2 f)|^{2} \mathcal{N}(f, 2k-2)  \\
\ll & (\log \kappa) e^{\frac {E2^{m/2}}{m}} \exp \left(\frac {4}{\alpha_j} \right)
 \sum^{ \mathcal{J}}_{u=j+1} \mathop{{\sum}^{h}}\limits_{\substack{f \in \mathcal{S}(j)}}\Big (2^{m/10}|P_m(f)| \Big )^{2\lceil 2^{m/2}\rceil }E_{\ell_1}( 2{\mathcal M}'_{1,j}(f)) \\
& \times \prod^{j}_{l=2} E_{\ell_l}( 2{\mathcal M}_{l,j}(f)){\mathcal N}_{l}(f, 2k-2)\prod^{\mathcal J}_{l=j+2} {\mathcal N}_{l}(f, 2k-2)  \\
& \times \Big ( \Big ( \frac {2 \cdot 10^3}{\ell_{j+1}}|{\mathcal M}_{j+1,u}(f)|\Big )^{2\lceil 1/(10^3\alpha_{j+1})\rceil }+ \Big( \frac {120 |{\mathcal P}_{j+1}(f)|}{\ell_{j+1}}\Big)^{2\lceil 1/(10^3\alpha_{j+1})\rceil} \Big ( \frac {10^3}{\ell_{j+1}}|{\mathcal M}_{j+1,u}(f)| \Big )^{2\lceil 1/(10^3\alpha_{j+1})\rceil }\Big ) \\
\ll & (\log \kappa)e^{\frac {E2^{m/2}}{m}} \exp \left(\frac {4}{\alpha_j} \right)
 \sum^{ \mathcal{J}}_{u=j+1} \mathop{{\sum}^{h}}\limits_{\substack{f \in \mathcal{S}(j)}}\Big (2^{m/10}|P_m(f)| \Big )^{2\lceil 2^{m/2}\rceil }E_{\ell_1}( 2{\mathcal M}'_{1,j}(f)) \\
& \times \prod^{j}_{l=2} E_{\ell_l}( 2{\mathcal M}_{l,j}(f)){\mathcal N}_{l}(f, 2k-2)\prod^{\mathcal J}_{l=j+2} {\mathcal N}_{l}(f, 2k-2)  \\
& \times \Big ( \Big (  \frac {2 \cdot 10^3}{\ell_{j+1}}|{\mathcal M}_{j+1,u}(f)|\Big )^{2\lceil 1/(10^3\alpha_{j+1})\rceil } +\Big( \frac {120 |{\mathcal P}_{j+1}(f)|}{\ell_{j+1}}\Big)^{4\lceil 1/(10^3\alpha_{j+1})\rceil} +\Big (  \frac {10^3}{\ell_{j+1}}|{\mathcal M}_{j+1,u}(f)|\Big )^{4\lceil 1/(10^3\alpha_{j+1})\rceil }\Big ).
\end{split}
\end{align*}

  As the treatments are similar, it suffices to estimate the expression given by
\begin{align}
\label{Sdef}
\begin{split}
 S:=& (\log \kappa)e^{\frac {E2^{m/2}}{m}} \exp \left(\frac {4}{\alpha_j} \right)
 \sum^{ \mathcal{J}}_{u=j+1} \mathop{{\sum}^{h}}\limits_{\substack{f \in \mathcal{S}(j)}}\Big (2^{m/10}|P_m(f)| \Big )^{2\lceil 2^{m/2}\rceil }E_{\ell_1}( 2{\mathcal M}'_{1,j}(f)) \\
& \times \prod^{j}_{l=2} E_{\ell_l}( 2{\mathcal M}_{l,j}(f)){\mathcal N}_{l}(f, 2k-2)\prod^{\mathcal J}_{l=j+2} {\mathcal N}_{l}(f, 2k-2) \times \Big (   \frac {2 \cdot 10^3}{\ell_{j+1}}|{\mathcal M}_{j+1,u}(f)| \Big )^{2\lceil 1/(10^3\alpha_{j+1})\rceil }  \\
:=&  (\log \kappa)e^{\frac {E2^{m/2}}{m}} \exp \left(\frac {4}{\alpha_j} \right)\sum^{ \mathcal{J}}_{u=j+1}S_u.
\end{split}
\end{align}

  It remains to evaluate $S_u$ for a fixed $u$ and we do so by expanding the factors involved in $S_u$ into Dirichlet series. Note that such a series for $\Big (2^{m/10}|P_m(f)| \Big )^{2\lceil 2^{m/2}\rceil }$ is already given in \eqref{Pexpression}, a series for $\Big (   \frac {2 \cdot 10^3}{\ell_{j+1}}|{\mathcal M}_{j+1,u}(f)|\Big )^{2\lceil 1/(10^3\alpha_{j+1})\rceil }$ can be obtained similar to that given in \eqref{M1fbound}. Moreover, a series for $E_{\ell_1}( 2{\mathcal M}'_{1,j}(f))\prod^{j}_{l=2} E_{\ell_l}( 2{\mathcal M}_{l,j}(f)){\mathcal N}_{l}(f, 2k-2)\prod^{\mathcal J}_{l=j+2} {\mathcal N}_{l}(f, 2k-2)$ can obtained analogue to \eqref{Nexpression}. We further note that as $0 \leq m \leq (3/\log 2)\log\log\log \kappa$, these series can be written for simplicity as
\begin{align}
\label{factorseries}
\begin{split}
 & \Big (2^{m/10}|P_m(f)| \Big )^{2\lceil 2^{m/2}\rceil }= \sum_{n_1  \leq e^{(\log \log \kappa)^2}} u_{n_1}\eta(n_1), \\
 & \Big (   \frac {2 \cdot 10^3}{\ell_{j+1}}|{\mathcal M}_{j+1,u}(f)| \Big )^{2\lceil 1/(10^3\alpha_{j+1})\rceil }= \sum_{n_2  \leq \kappa^{4/10^3}} u_{n_2}\widetilde{\lambda}(n_2), \\
 & E_{\ell_1}( 2{\mathcal M}'_{1,j}(f))\prod^{j}_{l=2} E_{\ell_l}( 2{\mathcal M}_{l,j}(f)){\mathcal N}_{l}(f, 2k-2)\prod^{\mathcal J}_{l=j+2} {\mathcal N}_{l}(f, 2k-2) = \sum_{n_3  \leq \kappa^{40e^A 10^{-M/4}}} u_{n_3} \widetilde{\lambda}(n_3),
\end{split}
\end{align}
   where $|u_{n_i}| \ll  \kappa^{\varepsilon}, 1 \leq i \leq 3$.

  We multiply the above Dirichlet series together, noting that $n_i, 1 \leq i \leq 3$ are mutually co-prime. Then using the expansions for $\widetilde{\lambda}(n), \eta(n)$ given in \eqref{lambdaf} and \eqref{eta} and further using the estimations for the coefficients involved in these expansions given in \eqref{cnt}, \eqref{dnt}, we see that we may write $S_u$ as
\begin{align}
\label{Suexpression}
 S_u= \sum_{n  \leq \kappa^{1/10}} v_n  \sum_{t|n^2}e_n(t)\mathop{{\sum}^{h}}\limits_{\substack{f \in \mathcal{S}(j)}}\lambda_f(t^2),
\end{align}
  where $|v_{n}| \ll  \kappa^{\varepsilon}$ and
\begin{align*}
%%\label{Nexpression}
  \sum_{t|n^2}|e_n(t)| \leq 8^{\Omega(n)}.
\end{align*}

  We apply \eqref{Deltaest} to evaluate the last sum in \eqref{Suexpression}. Notice that we have $t^2 \ll \kappa$ for all $t$ involved in the sum. It follows from this and our discussions above that the contribution from the error term in \eqref{Deltaest} is
\begin{align*}
%%\label{Suexpression}
 \ll e^{-\kappa} \sum_{n  \leq \kappa^{1/10}} \kappa^{\varepsilon} \sum_{t|n^2}|e_n(t)| \ll \kappa^{-1},
\end{align*}
  which is negligible.

 It then suffices to consider on the main term contribution from  \eqref{Deltaest}, which equals
\begin{align*}
%%\label{Suexpression}
  \sum_{n  \leq \kappa^{1/10}} v_n \sum_{t|n^2}e_n(1).
\end{align*}
  Observing that $e_n(1)$ is multiplicative, so that we have by \eqref{factorseries}, \eqref{lambdaf} and \eqref{eta},
\begin{align*}
%%\label{Suexpression}
  \sum_{n  \leq \kappa^{1/10}} v_n \sum_{t|n^2}e_n(1)=(\sum_{n_1} u_{n_1}d_{n_1}(1))(\sum_{n_2} u_{n_2}c_{n_2}(1))(\sum_{n_3} u_{n_3}c_{n_3}(1)).
\end{align*}

  Using arguments that lead to estimations given in \eqref{Pmest} and \eqref{S0upperbound}, we see that
\begin{align}
\label{Su1stestmain}
\begin{split}
   \sum_{n_1} u_{n_1}d_{n_1}(1) \ll & 2^{-2^{m/2}}, \\
 \sum_{n_2} u_{n_2}c_{n_2}(1) \ll & e^{-10^3/\alpha_{j+1}}.
\end{split}
\end{align}

  It remains to evaluate $\sum_{n_3} u_{n_3}c_{n_3}(1)$. By the multiplicity of $c_n(1)$ again, we see that for some real numbers $\tilde{v}_{n_l}, 1 \leq l \leq \mathcal{J}$,
\begin{align*}
%%\label{Suexpression}
   \sum_{n_3} u_{n_3}c_{n_3}(1)=\prod_{l} (\sum_{n_l}\tilde{v}_{n_l}b_l(n_l)c_{n_1}(1)).
\end{align*}

   We examine the factor $\sum\limits_{n_l}\tilde{v}_{n_l}b_l(n_l)c_{n_1}(1)$ in the above expression that arises from the product
\begin{align*}
%%\label{Suexpression}
   E_{\ell_l}( 2{\mathcal M}_{l,j}(f)){\mathcal N}_{l}(f, 2k-2),
\end{align*}
  for some $2 \leq l \leq j$. Similar to \eqref{5.1}, we write
\begin{align*}
%%\label{Nexpression}
 E_{\ell_l}( 2{\mathcal M}_{l,j}(f))= \sum_{n} \frac{2^{\Omega(n)} s(n, \kappa^{\alpha_j})}{w(n)\sqrt{n}}b_{l}(n) \widetilde{\lambda}(n), \quad
 {\mathcal N}_{l}(f, 2k-2)= \sum_{n} \frac{(2k-2)^{\Omega(n)}}{w(n)\sqrt{n}}b_{l}(n) \widetilde{\lambda}(n).
\end{align*}

  We now extract the sum over terms involving with $c_n(1)$ in $ E_{\ell_l}( 2{\mathcal M}_{l,j}(f)){\mathcal N}_{l}(f, 2k-2)$ using the above expressions to see that it equals
\begin{align}
\label{sumsqurei}
 \sum_{n, n'}\frac{2^{\Omega(n)}(2k-2)^{\Omega(n')}s(n, \kappa^{\alpha_j})b_l(n)b_l(n')}{w(n)w(n')\sqrt{nn'}}c_{nn'}(1).
\end{align}
   Note that the factor $b_l(n)$ restricts $n$ to have
  all prime factors in $P_l$ such that $\Omega(n) \leq \ell_l$. If we remove the restriction on $\Omega(n)$, then the sum in \eqref{sumsqurei} becomes
\begin{align}
\label{sumsqureififree}
 \sum_{\substack{n, n' \\ p | n \Rightarrow p \in P_l}}\frac{2^{\Omega(n)}(2k-2)^{\Omega(n')}s(n, \kappa^{\alpha_j})b_l(n')}{w(n)w(n')\sqrt{nn'}}c_{nn'}(1).
\end{align}
     On the other hand, using Rankin's trick by noticing that $2^{n-\ell_l}\ge 1$ if $\Omega(n) > \ell_l$, we see that the error introduced this way does not exceed
\begin{align}
\label{sumsqureierror1}
 \sum_{\substack{n, n' \\ p | nn' \Rightarrow p \in P_l}}\frac{2^{\Omega(n)-\ell_l}2^{\Omega(n)}|2k-2|^{\Omega(n')}s(n, \kappa^{\alpha_j})}{w(n)w(n')\sqrt{nn'}}|c_{nn'}(1)|.
\end{align}

   Similarly, we may remove the restriction of $b_l(n')$ on $\Omega(n')$ to further write the expression in \eqref{sumsqureififree} as
\begin{align}
\label{sumsqureifif'ifree}
 \sum_{\substack{n, n' \\ p | nn' \Rightarrow p \in P_l}}\frac{2^{\Omega(n)}(2k-2)^{\Omega(n)}s(n, \kappa^{\alpha_j})}{w(n)w(n')\sqrt{nn'}}c_{nn'}(1)+O\Big(\sum_{\substack{n, n' \\ p | nn' \Rightarrow p \in P_l}}\frac{2^{\Omega(n')-\ell_l}2^{\Omega(n)}|2k-2|^{\Omega(n')}s(n, \kappa^{\alpha_j})}{w(n)w(n')\sqrt{nn'}}|c_{nn'}(1)|\Big ).
\end{align}

    Note that both the main term and the error term  above as well as the expression in \eqref{sumsqureierror1} are now multiplicative functions of $n, n'$ and hence can be evaluated in terms of products over primes. We then apply Lemma \ref{RS}, \eqref{cvalue},  the estimations $|c_n(1)| \leq 3^{\Omega(n)}$ and $|s(n, \kappa^{\alpha_j})| \leq 1$ for all $n$ involved,  to see that
\begin{align*}
%%\label{sumsqureifif'ifreemain}
\begin{split}
 \sum_{\substack{n, n' \\ p | nn' \Rightarrow p \in P_l}}\frac{2^{\Omega(n)}(2k-2)^{\Omega(n)}s(n, \kappa^{\alpha_j})}{w(n)w(n')\sqrt{nn'}}c_{nn'}(1)=& \prod_{p \in  P_l }\Big  (1+\frac 12(\frac {2s(p, \kappa^{\alpha_j})}{\sqrt{p}}+\frac {2k-2}{\sqrt{p}})^2+O\Big( \frac 1{p^{3/2}} \Big) \Big ) \\
 =& \prod_{p \in  P_l }\Big  (1+\frac {2k^2}{p}+O\Big( \frac {\log p}{p\log \kappa^{\alpha_j}}+\frac 1{p^{3/2}}  \Big) \Big ),
\end{split}
\end{align*}
  where the last expression above follows by observing that for $p \leq \kappa^{\alpha_j}$,
\begin{align*}
%%\label{sumsqureifif'ifreemain}
  s(p, \kappa^{\alpha_j})=1+O\Big( \frac {\log p}{\log \kappa^{\alpha_j}} \Big ).
\end{align*}

    We further apply the estimation $1+x \leq e^x$ for any real $x$ to see that
\begin{align*}
%%\label{sumsqureifif'ifreemain}
\begin{split}
  \prod_{p \in  P_l }\Big  (1+\frac {2k^2}{p}+O\Big( \frac {\log p}{p\log \kappa^{\alpha_j}}+\frac 1{p^{3/2}}  \Big) \Big )
\leq \exp \Big ( \sum_{p \in  P_l } \frac {2k^2}{p}+O\Big( \sum_{p \in  P_l } \Big (\frac {\log p}{p\log \kappa^{\alpha_j}}+\frac 1{p^{3/2}}  \Big) \Big)\Big ).
\end{split}
\end{align*}

   We evaluate the error term in \eqref{sumsqureifif'ifree} and the expression in \eqref{sumsqureierror1} similarly to see via \eqref{sumpj} that they both are
\begin{align*}
%%\label{sumsqureifif'ifreemain}
 \leq 2^{-\ell_l/2}\exp \Big ( \sum_{p \in  P_l } \frac {2k^2}{p}+O\Big( \sum_{p \in  P_l } \Big (\frac {\log p}{p\log \kappa^{\alpha_j}}+\frac 1{p^{3/2}}  \Big) \Big)\Big ).
\end{align*}

   It follows that the contribution from the sum over terms involving with $c_n(1)$ in $ E_{\ell_l}( 2{\mathcal M}_{l,j}(f)){\mathcal N}_{l}(f, 2k-2)$ is
\begin{align*}
%%\label{sumsqureifif'ifreemain}
 \leq \Big (1+O(2^{-\ell_l/2})\Big )\exp \Big ( \sum_{p \in  P_l } \frac {2k^2}{p}+O\Big( \sum_{p \in  P_l } \Big (\frac {\log p}{p\log \kappa^{\alpha_j}}+\frac 1{p^{3/2}}  \Big) \Big)\Big ).
\end{align*}

 In the same manner, we obtain a similar estimation for the contribution from the sum over terms involving with $c_n(1)$ in $E_{\ell_1}( 2{\mathcal M}'_{1,j}(f))$.  Also, the contribution from the sum over terms involving with $c_n(1)$ in ${\mathcal N}_{l}(f, 2k-2)$ for $j+2 \leq l \leq \mathcal{J}$ is
\begin{align*}
%%\label{sumsqureifif'ifreemain}
 \leq \Big (1+O(2^{-\ell_l/2})\Big )\exp \Big  (\prod_{p \in  P_l } \Big (\frac {(2k-2)^2}{2p}+O\Big( \frac 1{p^{3/2}} \Big) \Big )\Big ).
\end{align*}

  The above estimations allow us to see that
\begin{align}
\label{prodEestmain}
\begin{split}
  & \sum_{n_3} u_{n_3}c_{n_3}(1) \\
\leq & \prod^j_{l=1}\Big (1+O(2^{-\ell_l/2})\Big ) \prod^{\mathcal{J}}_{l=j+2}\Big (1+O(2^{-\ell_l/2})\Big ) \times \exp \Big (\sum_{p \in  \bigcup^j_{l=1}P_l }\frac {2k^2}{p}
 +O\Big(\sum_{p \in  \bigcup^j_{l=1}P_l }\Big (\frac {\log p}{p\log \kappa^{\alpha_j}}+\frac 1{p^{3/2}}\Big ) \Big )\Big ) \\
& \times \exp \Big (\sum_{p \in  \bigcup^{\mathcal{J}}_{l=j+2}P_l }\frac {(2k-2)^2}{2p}
 +O\Big (\sum_{p \in  \bigcup^j_{l=1}P_l }\frac 1{p^{3/2}}\Big)\Big ) \\
\ll & \exp \Big (\sum_{p \in  \bigcup^j_{l=1}P_l }\frac {2k^2}{p}
 +O\Big(\sum_{p \in  \bigcup^j_{l=1}P_l }\Big (\frac {\log p}{p\log \kappa^{\alpha_j}}+\frac 1{p^{3/2}}\Big ) \Big )\Big ) \times \exp \Big (\sum_{p \in  \bigcup^{\mathcal{J}}_{l=j+2}P_l }\frac {(2k-2)^2}{2p}
 +O\Big (\sum_{p \in  \bigcup^j_{l=1}P_l }\frac 1{p^{3/2}}\Big)\Big ).
\end{split}
\end{align}

    Note that we have
\begin{align*}
%%\label{sumsqureifif'ifreemain}
  \sum_{p \in  \bigcup^{\mathcal{J}}_{l=1}P_l }\frac 1{p^{3/2}} \ll 1.
\end{align*}
  Moreover, by Lemma \ref{RS}, we have
\begin{align*}
%%\label{sumsqureifif'ifreemain}
  \sum_{p \in  \bigcup^j_{l=1}P_l }\frac {\log p}{p\log \kappa^{\alpha_j}}=& \sum_{p \leq  \kappa^{\alpha_j} }\frac {\log p}{p\log \kappa^{\alpha_j}} \ll 1 \\
\sum_{p \in  \bigcup^{\mathcal{J}}_{l=j+2}P_l }\frac {(2k-2)^2}{2p}=& \sum_{p \in  \bigcup^{\mathcal{J}}_{l=j+2}P_l }\frac {2k^2}{p}+\sum_{p \in  \bigcup^{\mathcal{J}}_{l=j+2}P_l }\frac {(2k-2)^2-(2k)^2}{2p} \\
\leq &  \sum_{p \in  \bigcup^{\mathcal{J}}_{l=j+2}P_l }\frac {2k^2}{p}+2(2-4k)\log \frac {1}{\alpha_{j+1}}+O(1).
\end{align*}
  We thus conclude from the above that upon taking $M$ large enough, the last expression in \eqref{prodEestmain} is
\begin{align}
\label{Su2ndestmain}
  \ll & e^{2(2-4k)\log \frac {1}{\alpha_{j+1}}}\exp \Big (\sum_{p \in  \bigcup^{\mathcal{J}}_{l=1}P_l }\frac {2k^2}{p}\Big ) \ll e^{\frac {1}{\alpha_{j+1}}}(\log \kappa)^{2k^2},
\end{align}
   where the last estimation above follows from Lemma \ref{RS}.

   It follows from \eqref{Su1stestmain} and \eqref{Su2ndestmain} that we have
\begin{align*}
%%\label{Su1stestmain}
\begin{split}
   S_u \ll 2^{-2^{m/2}} e^{-200/\alpha_{j+1}}(\log \kappa)^{2k^2}.
\end{split}
\end{align*}

  Applying the above estimation in \eqref{Sdef} and making use of the observation that $20/\alpha_{j+1}=1/\alpha_j$, we deduce that
\begin{align*}
%%\label{Sdef}
\begin{split}
 S \ll &  (\log \kappa)^{2k^2+1}e^{\frac {E2^{m/2}}{m}}2^{-2^{m/2}} \exp \left(-\frac {2}{\alpha_j} \right)(\mathcal{J}-j).
\end{split}
\end{align*}

  We further apply the above estimation in \eqref{upperboundprodE0} to see that
\begin{align}
\label{upperboundprodLN}
\begin{split}
 \mathop{{\sum}^{h}}\limits_{\substack{f \in \mathcal{S}(j)\bigcap \mathcal{P}(m)\bigcap \mathcal{T}}}|L(\frac 12, \operatorname{sym}^2 f)|^{2} \mathcal{N}(f, 2k-2)
\ll &  (\log \kappa)^{2k^2+1}e^{\frac {E2^{m/2}}{m}}2^{-2^{m/2}} \exp \left(-\frac {2}{\alpha_j} \right)(\mathcal{J}-j) \\
\ll  & (\log \kappa)^{2k^2+1}e^{\frac {E2^{m/2}}{m}}2^{-2^{m/2}} \exp \left(-\frac {1}{\alpha_j} \right),
\end{split}
\end{align}
 where the last estimation above follows from the observation that we have for $1 \leq j \leq \mathcal{J}-1$,
$$ \mathcal{J}-j \leq \frac{\log(1/\alpha_{j})}{\log 20}. $$

 Now summing the last expression of \eqref{upperboundprodLN} over $j$ and $m$ and noticing that these sums are convergent, we derive the desired estimation in \eqref{sumovermj} and hence completes the proof of Proposition \ref{Prop6}.

\vspace*{.5cm}

\noindent{\bf Acknowledgments.} The author would like to thank R. Khan for some helpful suggestions.
 The author is supported in part by NSFC grant 11871082.
\bibliography{biblio}
\bibliographystyle{amsxport}

\vspace*{.5cm}

\end{document}